\documentclass[12pt]{article}

\usepackage[margin=1.17in]{geometry}

\usepackage{amssymb,amsmath}
\usepackage{amsthm,enumitem}
\usepackage{hyperref}
\usepackage{mathrsfs}
\usepackage{textcomp}
\hypersetup{
	colorlinks,%
	citecolor=black,%
	filecolor=black,%
	linkcolor=black,%
	urlcolor=black
}
\usepackage[abbrev,nobysame]{amsrefs}
\usepackage{todonotes}

\usepackage{verbatim}

\usepackage{sectsty}

\makeatletter

\newdimen\bibspace
\makeatother

\newtheorem{Theorem}{Theorem}[section]
\newtheorem*{Theorem*}{Theorem A}
\newtheorem{Lemma}[Theorem]{Lemma}
\newtheorem{Proposition}[Theorem]{Proposition}
\newtheorem*{Assumption*}{Assumption (H)}
\newtheorem*{Assumption**}{Assumption (G)}
\newtheorem{Definition}{Definition}[section]

\newtheorem{Corollary}[Theorem]{Corollary}
\newtheorem{Remark}[Theorem]{Remark}

\newcommand{\R}{\mathbb{R}}

\def\d{\operatorname{dist}(x,\partial\Omega)}
\def \({\left(}
\def \){\right)}

\allowdisplaybreaks

\begin{document}
\title{\textbf{Boundary regularity for fractional porous medium-type equations in bounded domains}\bigskip}
	
\author{\medskip  Mengyu Huo,\ \ Tianling Jin\footnote{T. Jin is partially supported by NSFC grant 12122120, and Hong Kong RGC grants GRF 16304125, GRF 16303624  and GRF 16303822.}, \  \
Jingang Xiong\footnote{J. Xiong is partially supported by NSFC grant 12325104.}}

\date{\today}

\maketitle

\begin{abstract}
We prove boundary H\"older regularity for fractional porous medium-type equations in bounded smooth domains with zero exterior data. For $0<\sigma<1$, $n>2\sigma$, and $m>1$, we show that nonnegative weak dual solutions $v$ of $\partial_t v+(-\Delta)^\sigma(v^m)=0$ satisfy
\[
\frac{v}{d^{\sigma/m}}\in
C_{x,t}^{\alpha,\frac{\alpha m}{\sigma(m+1)}}
\]
up to the boundary and away from the initial time, for some $\alpha\in(0,\sigma)$, where $d$ denotes the distance to the boundary. More generally, we establish boundary H\"older estimates for $u/d^\sigma$ for nonnegative viscosity solutions of $\partial_t u=u^\gamma Lu$, with $0<\gamma<1$ and initial data comparable to $d^\sigma$. Here $L$ is a symmetric stable operator of order $2\sigma$ whose measurable angular density is bounded above and below by positive constants. Our proof combines comparison principles and boundary barriers with a blow-up and compactness argument based on a Liouville theorem for ancient solutions in a half-space.


\medskip

\end{abstract}

\section{Introduction}
\subsection{The fractional porous medium equation}
Let $\Omega\subset\mathbb R^n$ be a bounded smooth domain, and
\[
0<\sigma<1, \quad n>2\sigma, \quad m>1.
\]
We consider the fractional porous medium equation
\begin{equation}\label{eq:fpme0}
\begin{cases}
\partial_t v=-(-\Delta)^\sigma v^m & \text{in }\Omega\times(0,+\infty),\\
v=0 & \text{in }\Omega^c\times (0,+\infty)\\
v=v_0 & \text{on }\Omega\times\{0\},
\end{cases}
\end{equation}
where $(-\Delta)^\sigma$ is the fractional Laplacian defined by
\[
(-\Delta)^\sigma f(x)=c_{n,\sigma}\,\operatorname{P.V.}\int_{\mathbb R^n}
\frac{f(x)-f(y)}{|x-y|^{n+2\sigma}}\,dy,
\]
with the usual normalization constant $c_{n,\sigma}>0$. 

The known estimates give the boundary behavior $v\asymp d^{\sigma/m}$ at positive times. We prove that the quotient $v^m/d^\sigma$ is H\"older continuous up to the boundary in both space and time. We first recall the notion of weak dual solutions to \eqref{eq:fpme0}. Let $G_\Omega$ be the Green function of $(-\Delta)^\sigma$ in $\Omega$ with zero exterior condition in $\mathbb R^n\setminus\Omega$. The inverse Dirichlet operator is given by
\[
(-\Delta)^{-\sigma}f(x)=\int_\Omega G_\Omega(x,y)f(y)\,dy.
\]
Let $\Phi_1$ be the positive first eigenfunction of $(-\Delta)^\sigma$ with zero exterior data, normalized by $\|\Phi_1\|_{L^2(\Omega)}=1$:
\[
(-\Delta)^\sigma\Phi_1=\lambda_1\Phi_1\quad\text{in }\Omega,
\qquad \Phi_1=0\quad\text{in }\Omega^c,
\qquad \lambda_1>0.
\]
Writing $$d(x)=\operatorname{dist}(x,\partial\Omega),$$ we have $\Phi_1\asymp d^\sigma$; see \cite{BonforteVazquez2016}. Define
\[
L^1_{\Phi_1}(\Omega)=\left\{f:\Omega\to\mathbb R\text{ measurable}:\int_\Omega |f|\Phi_1\,dx<\infty\right\},
\]
with the norm
\[
\|f\|_{L^1_{\Phi_1}(\Omega)}=\int_\Omega |f|\Phi_1\,dx.
\]
The notion of weak dual solution was introduced by Bonforte and V\'azquez in \cite{BonforteVazquez2015} and developed further in \cite{BonforteVazquez2016}. We use its formulation in Definitions~5.1 and~5.2 of \cite{BonforteFigalliRosOton2017}.

\begin{Definition}[Weak dual solution]\label{def:weak-dual}
Let $T>0$. A nonnegative function $v$ in $\Omega$, extended by zero to $\Omega^c$, is a weak dual solution of 
\begin{equation}\label{eq:fpmeT}
\begin{cases}
\partial_t v=-(-\Delta)^\sigma v^m & \text{in }\Omega\times(0,T),\\
v=0 & \text{in }\Omega^c\times (0,T)
\end{cases}
\end{equation}
 if $v\in C\big((0,T);L^1_{\Phi_1}(\Omega)\big)$, $v^m\in L^1\big((0,T);L^1_{\Phi_1}(\Omega)\big)$,
and
\begin{equation}\label{eq:weak-dual}
\int_0^T\!\int_\Omega (-\Delta)^{-\sigma}v\,\partial_t\psi\,dx\,dt
=\int_0^T\!\int_\Omega v^m\psi\,dx\,dt
\end{equation}
for every test function $\psi$ such that $\psi/\Phi_1\in C_c^1\big((0,T);L^\infty(\Omega)\big).$

It solves the Cauchy--Dirichlet problem \eqref{eq:fpmeT} with initial datum $0\leq v_0\in L^1_{\Phi_1}(\Omega)$ if, in addition, $v\in C\big([0,T);L^1_{\Phi_1}(\Omega)\big)$,
and $v(\cdot,0)=v_0\quad\text{in }L^1_{\Phi_1}(\Omega)$.
\end{Definition}

Thus, \eqref{eq:weak-dual} expresses $\partial_t((-\Delta)^{-\sigma}v)=-v^m$ weakly, with the boundary condition encoded in $(-\Delta)^{-\sigma}$. For every nonnegative initial datum $v_0\in L^1_{\Phi_1}(\Omega)$, there exists a unique minimal weak dual solution, obtained by increasing approximation with mild solutions. It belongs to $C([0,\infty);L^1_{\Phi_1}(\Omega))$ and satisfies the comparison principle; see Theorems~4.4 and~4.5 of \cite{BonforteVazquez2016}.

For nontrivial $v_0$, Bonforte, Figalli, and Ros-Oton \cite{BonforteFigalliRosOton2017} established the global Harnack principle for all positive times: setting
\[
t_*=c_*\|v_0\|_{L^1_{\Phi_1}(\Omega)}^{-(m-1)},
\]
one has
\begin{equation}\label{eq:dual-global-harnack}
\underline\kappa\min\left\{1,\frac{t}{t_*}\right\}^{\frac{m}{m-1}}
\frac{d(x)^{\sigma/m}}{t^{1/(m-1)}}
\leq v(x,t)\leq
\overline\kappa\frac{d(x)^{\sigma/m}}{t^{1/(m-1)}},
\qquad x\in\Omega,\ t>0,
\end{equation}
where $c_*,\underline\kappa,\overline\kappa>0$ depend only on $n,\sigma,m,\Omega$. In particular, $v$ is positive in $\Omega$ at every positive time. Theorems~1.4--1.6 of \cite{BonforteFigalliRosOton2017} also give, for every $0<t_0<t_1$, H\"older regularity up to $\partial\Omega$,
\[
v\in C^{\frac{\sigma}{m},\frac{1}{2m}}_{x,t}(\overline\Omega\times[t_0,t_1]),
\]
with optimal spatial exponent $\sigma/m$, and higher interior spatial regularity:
\[
v\in C_x^\infty(\Omega\times(0,\infty)),
\]
\[
|D_x^k v(x,t)|\leq C_k d(x)^{\sigma/m-k},\quad k\geq1,\ t\geq t_0.
\]
Here $C_k$ depends only on $n,\sigma,m,k,\Omega,t_0$ and $\|v_0\|_{L^1_{\Phi_1}(\Omega)}$. Moreover, for every compact $K\Subset\Omega$,
\[
v\in C_t^{1,\alpha}(K\times[t_0,t_1])
\quad\text{for some }\alpha=\alpha(\sigma,m)>0.
\]
Thus, at positive times, these weak dual solutions are classical and positive in the interior and H\"older continuous up to the boundary $\partial\Omega$.

In this paper, we refine the boundary regularity and prove that $v^m/d^\sigma$ is H\"older continuous in $\overline\Omega\times[t_0,+\infty)$.

\begin{Theorem}\label{thm:fpme}
Let $0\leq v_0\in L^1_{\Phi_1}(\Omega)$, and let $v$ be the weak dual solution of \eqref{eq:fpme0} described above. For every $t_0>0$, there exist $\alpha\in(0,\sigma)$ and $C>0$, depending only on $n,\sigma,m,\Omega,t_0$ and $\|v_0\|_{L^1_{\Phi_1}(\Omega)}$, such that, for every $t_1>t_0$,
\[
\left\|\frac{v}{d^{\sigma/m}}\right\|_{C_{x,t}^{\alpha,\frac{\alpha m}{\sigma(m+1)}}(\overline\Omega\times[t_0,t_1])}\leq C.
\]
\end{Theorem}

Our proof combines a compactness argument with a Liouville theorem. We work in the viscosity framework to carry out this argument.

\subsection{A class of degenerate nonlocal equations}
Since nontrivial weak dual solutions $v$ are positive and classical in the interior for positive times, the change of variables $u=v^m$ transforms $v_t=-(-\Delta)^\sigma v^m$ into
\begin{equation}\label{eq:tfpme}
u_t=-mu^\gamma(-\Delta)^\sigma u,
\end{equation}
where $\gamma=1-\frac{1}{m}\in(0,1)$. We study the regularity of solutions to equations of this form for a more general class of nonlocal operators.

Fix $T>0$ and $\gamma\in(0,1)$, and let $u_0$ be positive in $\Omega$. We consider the Cauchy--Dirichlet problem
\begin{equation}\label{eq:main-problem}
\begin{cases}
\partial_t u=u^\gamma Lu & \text{in }\Omega\times(-T,0],\\
u=0 & \text{in }\Omega^c\times[-T,0],\\
u=u_0 & \text{on }\Omega\times\{-T\}.
\end{cases}
\end{equation}
Here $L$ is a linear, translation-invariant nonlocal operator of order $2\sigma$, with $\sigma\in(0,1)$, in the class
$\mathcal L_*$ considered by Ros-Oton and Serra \cite{ros2016boundary}. It is given by
\[
Lu(x)=\int_{\mathbb R^n}
\left(\frac{u(x+z)+u(x-z)}{2}-u(x)\right)
\frac{\mu(z/|z|)}{|z|^{n+2\sigma}}\,dz,
\]
where
\[
\mu\in L^\infty(\mathbb S^{n-1}),\qquad
\mu(\theta)=\mu(-\theta),\qquad
0<\lambda\leq\mu\leq\Lambda<+\infty.
\]
When $\mu$ is constant, $L$ is a positive multiple of $-(-\Delta)^\sigma$. 

The class $\mathcal L_0$, introduced by Caffarelli and Silvestre
\cite{luis}, consists of the operators obtained by replacing
$\mu(z/|z|)/|z|^{n+2\sigma}$ in the formula above by an even measurable
kernel $K$ satisfying
\[
\frac{\lambda}{|z|^{n+2\sigma}}\leq K(z)
\leq\frac{\Lambda}{|z|^{n+2\sigma}},\qquad z\neq0.
\]
Thus, $\mathcal L_*\subseteq\mathcal L_0$.

We study the behavior and regularity of $u$ near $\partial\Omega$. The choice of $\mathcal L_*$ reflects the expected boundary behavior: every $L\in\mathcal L_*$ satisfies
\[
L((x_n)_+^\sigma)=0
\qquad\text{in }\{x_n>0\}.
\]
Thus, $(x_n)_+^\sigma$ is a common half-space solution for all operators in $\mathcal L_*$. This property allows the construction of barriers comparable to $d^\sigma$ and motivates the study of the quotient $u/d^\sigma$. Ros-Oton and Serra \cite{ros2016boundary} proved a H\"older estimate for this quotient near a flat boundary for bounded solutions of the extremal inequalities associated with $\mathcal L_*$. They also showed that the larger class $\mathcal L_0$ does not share a common boundary power: the associated extremal equations admit solutions with boundary powers different from $\sigma$. Hence, ellipticity with respect to $\mathcal L_0$ alone does not guarantee the boundary estimate $u\asymp d^\sigma$.

This boundary behavior also arises from the fractional porous medium equation. As mentioned earlier, Bonforte, Figalli, and Ros-Oton \cite{BonforteFigalliRosOton2017} proved that every nontrivial nonnegative weak dual solution $v$ of \eqref{eq:fpme0} satisfies $v(\cdot,t)\asymp d^{\sigma/m}$ for each $t>0$. The transformation $u=v^m$ therefore gives $u(\cdot,t)\asymp d^\sigma$. 

The estimate $u\asymp d^\sigma$ provides upper and lower bounds for $u/d^\sigma$, but does not ensure that the quotient has a limit at the boundary. Likewise, the $C^{\sigma/m}$ estimate for $v$ does not imply continuity of $v/d^{\sigma/m}$. To describe the boundary behavior more precisely, we seek a modulus of continuity for the quotient. For the classical porous medium equation ($\sigma=1$), Jin, Ros-Oton, and Xiong \cite{JinRosOtonXiong2024} proved the H\"older continuity of $u/d$ after a waiting time. This estimate is a key step in their proof of optimal boundary regularity. We seek the corresponding H\"older estimate for $u/d^\sigma$ for the degenerate nonlocal equation \eqref{eq:main-problem}.

We consider nonnegative viscosity solutions of \eqref{eq:main-problem} with initial data $u_0\in C(\overline\Omega)$ satisfying
\begin{equation}\label{u_0}
c_1 \d^\sigma\leq u_0(x)\leq c_2 \d^\sigma, \qquad x\in\overline{\Omega},
\end{equation}
for some constants $c_1,c_2>0$. By \eqref{eq:dual-global-harnack}, this condition holds at each positive time for $u=v^m$ when $v$ is a nontrivial nonnegative weak dual solution of \eqref{eq:fpme0}. In the viscosity framework, we first prove that this two-sided boundary estimate persists in time: there exist constants $C_1,C_2>0$ such that
\[
C_1\d^\sigma\leq u(x,t)\leq C_2\d^\sigma, \qquad (x,t)\in\Omega\times[-T,0].
\]
The constants are uniform on this time interval.

The main obstacle to estimating the modulus of continuity of $u/d^\sigma$ is the degeneracy of the equation at the boundary. Since $u\asymp d^\sigma$, we have $u^\gamma\asymp d^{\sigma\gamma}$ and can write the equation in $\Omega$ as
\[
\partial_t u=d^{\sigma\gamma}a(x,t)Lu, \qquad a(x,t)=\left(\frac{u(x,t)}{d(x)^\sigma}\right)^\gamma,
\]
where $a$ is bounded above and below by positive constants. Although the equation is uniformly parabolic on compact subsets of $\Omega$, the factor $d^{\sigma\gamma}$ vanishes at the boundary. Moreover, H\"older continuity of $a$ depends on that of $u/d^\sigma$ and therefore cannot be assumed in proving the desired boundary estimates. Applying the elliptic estimates recalled above at a fixed time would require control of $u^{-\gamma}\partial_tu$ near the boundary.

The boundary behavior also determines the natural parabolic scaling. At spatial scale $r$, the solution $u$ has size $r^\sigma$, so the coefficient $u^\gamma$ has size $r^{\sigma\gamma}$. Since $L$ has order $2\sigma$, the corresponding time scale is $r^{\sigma(2-\gamma)}$. For $x_0\in\partial\Omega$ and $t_0\in(-T,0]$, the function
\[
u_r(x,t):=r^{-\sigma}u(x_0+rx,\,t_0+r^{\sigma(2-\gamma)}t)
\]
satisfies the same equation on the rescaled domain and time interval. The nonlocal operator still involves the values of $u_r$ outside the unit ball, so the estimates must control these contributions uniformly under rescaling.

Our main result establishes H\"older continuity of $u/d^\sigma$ up to the boundary. Under the assumptions above, we prove that
\[
\frac{u}{d^\sigma}\in C_{x,t}^{\alpha,\frac{\alpha}{\sigma(2-\gamma)}}(\overline{\Omega}\times (-T/2,0])
\]
for some $\alpha\in(0,\sigma)$. Together with the two-sided boundary estimate, this implies that $u/d^\sigma$ extends to a positive H\"older continuous function on $\overline\Omega$ at each such time. More precisely, we prove the following theorem.

\begin{Theorem}\label{mainthm}
Let $\Omega$ be a bounded $C^{1,1}$ domain, $L\in\mathcal L_*$, $T>0$, and $0<\gamma<1$. Let $u:\mathbb R^n\times[-T,0]\to[0,\infty)$ be continuous on $\overline\Omega\times[-T,0]$ and solve \eqref{eq:main-problem} in the viscosity sense. Assume that $u_0$ satisfies \eqref{u_0}. Then there exist $\alpha\in(0,\sigma)$ and $C>0$, depending only on $n,\sigma,\gamma,\lambda,\Lambda,c_1,c_2,T$ and $\Omega$, such that
\[
\left\|\frac{u}{d^\sigma}\right\|_{C_{x,t}^{\alpha,\frac{\alpha}{\sigma(2-\gamma)}}(\overline\Omega\times[-T/2,0])}\leq C.
\]
The quotient on the boundary is understood as its continuous extension from $\Omega$.
\end{Theorem}
The notion of viscosity solution used here is given in Definition~\ref{defn:vis}.

\subsection{Related work}
We first recall some results for the classical porous medium equation
\begin{equation}\label{eq:classical-problem}
\begin{cases}
\partial_t u=\Delta(u^m) & \text{in }\Omega\times(0,+\infty),\\
u=0 & \text{on }\partial \Omega\times(0,+\infty),\\
u(\cdot,0)=u_0 & \text{on }\Omega,
\end{cases}
\end{equation}
where $m>1$. Solutions to this equation exist globally in time. Aronson and Peletier \cite{AronsonPeletier1981} established spatial Lipschitz continuity after a waiting time and characterized the long-time behavior of solutions to \eqref{eq:classical-problem}, with a sharp convergence rate. Jin, Ros-Oton, and Xiong \cite{JinRosOtonXiong2024} proved optimal regularity and refined asymptotic results after a waiting time. For related fully nonlinear second-order equations, Kim and Lee \cite{KimLeeDegenerate2013} studied the long-time behavior, while Lee and Yun \cite{LeeYun2025} established boundary regularity estimates for fully nonlinear degenerate parabolic equations. Yun \cite{Yun2024} proved global $C^{2,\alpha}$ estimates for fully nonlinear equations of porous medium type. Boundary regularity for the classical fast diffusion equation, corresponding to $0<m<1$ in \eqref{eq:classical-problem}, has been studied in \cite{ChenDiBenedetto1988, DiBenedettoKwongVespri1991,JinXiong2023Boundary,JinXiong2025Dirichlet}. For the very fast diffusion regime $m<-1$, related boundary regularity results and estimates for the linearized equations have been obtained in \cite{JinTuXiongZheng2026Schauder, JinTuXiongZheng2026Regularity}.

Boundary regularity for linear nonlocal parabolic equations, corresponding to $\gamma=0$, has also been studied. For general stable operators, including those with singular spectral measures, Fern\'andez-Real and Ros-Oton \cite{FernandezRealRosOton2017Parabolic} proved that solutions of $\partial_tu-Lu=f$ with bounded $f$ and zero exterior data in bounded $C^{1,1}$ domains satisfy $u/d^\sigma\in C_{x,t}^{\sigma-\epsilon,(\sigma-\epsilon)/(2\sigma)}$ up to the boundary away from the initial time, for every $0<\epsilon<\sigma$. 

For nonlocal porous medium equations in bounded domains, the works \cite{BonforteVazquez2015,BonforteSireVazquez2015,BonforteFigalliVazquez2018} establish existence, a priori estimates, sharp global bounds, and long-time asymptotics. Bonforte, Figalli, and Ros-Oton \cite{BonforteFigalliRosOton2017} proved the sharp boundary estimates for the fractional porous medium equation recalled above. More recently, Bonforte and Fuertes-Mor\'an \cite{BonforteFuertesMoran2025} established sharp boundary estimates and Harnack inequalities for a broad class of fractional porous medium-type equations.

\subsection{Outline of the proof}

For the classical porous medium equation \eqref{eq:classical-problem}, Jin, Ros-Oton, and Xiong \cite{JinRosOtonXiong2024} established boundary H\"older regularity of the quotient through a De Giorgi iteration as in \cite{JinXiong2025Dirichlet}. In the nonlocal setting, the boundary degeneracy is coupled with interactions across the whole space. In particular, localization introduces tail terms, and comparison arguments require control of the functions throughout the exterior of the comparison region. We also allow operators with merely measurable angular densities. To address these difficulties, we use a blow-up and compactness argument based on boundary estimates with tail control and a Liouville theorem
in the half-space.

We first establish comparison principles for the nonlinear equation and the auxiliary extremal equations, allowing exterior values in the weighted class \eqref{eq:tail-class}. An elliptic profile and the associated separable parabolic barriers then yield $u\asymp d^\sigma$. These bounds identify the degeneracy $u^\gamma\asymp d^{\sigma\gamma}$ and the intrinsic time scale $r^{\sigma(2-\gamma)}$.

The next step is to study the corresponding extremal inequalities in the half-space. We construct boundary barriers and prove a Hopf lemma with error terms depending on the forcing and the negative tail. These estimates yield decay of the oscillation of $u/(x_n)^\sigma$, with uniform tail control at each scale. Rescaling this estimate gives a Liouville theorem for ancient solutions with controlled growth.

To pass from the half-space to $C^{1,1}$ domains, we prove boundary growth estimates for signed solutions. Together with the interior H\"older estimate and uniform tail bounds, these estimates provide compactness both locally uniformly and in the weighted $L^1$ norm. The latter convergence is needed to pass to the limit in the nonlocal equations. We subtract suitable multiples of the first eigenfunction and rescale the resulting errors. The Liouville theorem then yields a boundary expansion with a higher-order remainder. Comparing these expansions at nearby boundary points and applying interior estimates gives H\"older continuity of $u/d^\sigma$ in space and time.

The paper is organized as follows. Section~\ref{sec:vs} introduces the viscosity framework, recalls the interior estimates, and states the comparison principles, whose proofs are given in Appendix~\ref{app:comparison}. Section~\ref{sec:Lip} constructs the elliptic profile and uses separable parabolic barriers to establish the two-sided boundary bounds. In Section~\ref{sec:liouville}, we prove the boundary growth estimate and a Hopf lemma in the half-space, derive the oscillation estimate for solutions with controlled tails, and establish the Liouville theorem. Section~\ref{sec:Boundary estimate} develops boundary barriers in $C^{1,1}$ domains and proves the compactness result for rescaled solutions. Finally, Section~\ref{sec:ratio} proves the boundary expansion and H\"older regularity of the quotient, and applies the result to weak dual solutions of the fractional porous medium equation.

\medskip

\noindent\textbf{AI usage statement.} All mathematical results, the proof structure, and the main arguments were developed by the authors. In subsequent revisions, AI tools were used only to edit the language and organize the exposition. All AI-assisted modifications were independently reviewed and verified by the authors, who take full responsibility for the final manuscript. 

\section{Viscosity solutions and comparison principles}\label{sec:vs}
\subsection{Operators, notation, and viscosity solutions}
Unless otherwise stated, $\Omega$ is a bounded $C^{1,1}$ domain, and we write $d(x)=\operatorname{dist}(x,\partial\Omega)$ for $x\in\Omega$. We extend powers of the interior distance by zero when they are used as functions on $\R^n$. We use the cylinders
\[
Q_r(x_0,t_0)=B_r(x_0)\times(t_0-r^{\sigma(2-\gamma)},t_0],
\qquad Q_r=Q_r(0,0),
\]
and $Q_r^+=Q_r\cap(\{x_n>0\}\times\R)$. We use the superscript $+$ only for the half-space problem. In a curved domain we use $Q_r\cap(\Omega\times\R)$, including the part below a tangent plane. We write $$\Omega_T=\Omega\times(-T,0]$$ and
\[
\delta f(x,z)=\frac{f(x+z)+f(x-z)}2-f(x).
\]
All viscosity equations on a finite time interval are imposed on an interval of the form $(a,b]$. Initial data are prescribed at time $a$, and the equation is tested also at the terminal time $b$. All contacts in time are from the past. In the weak dual formulation in the introduction, test functions have compact support in the open time interval.

We impose the following integrability condition for the nonlocal terms. Functions are measurable on $\R^n$ at each time. When continuity is assumed, it is required on the closure of the space-time domain in question. On every compact time interval $I$ under consideration we require
\begin{equation}\label{eq:tail-class}
\int_{\R^n}\frac{|u(x,t)|}{1+|x|^{n+2\sigma}}\,dx<\infty,
\qquad
\int_{\R^n}\frac{|u(x,t)-u(x,t_0)|}{1+|x|^{n+2\sigma}}\,dx\longrightarrow0
\quad \mbox{as } t\to t_0\text{ in }I.
\end{equation}
This condition expresses continuity in the weighted $L^1$ norm; the exterior values need not be continuous or bounded in space. It implies
\begin{equation}\label{eq:uniform-tail}
\sup_{t\in I}\int_{\R^n}\frac{|u(x,t)|}{1+|x|^{n+2\sigma}}\,dx<\infty,
\qquad
\lim_{R\to\infty}\sup_{t\in I}\int_{|x|>R}\frac{|u(x,t)|}{1+|x|^{n+2\sigma}}\,dx=0.
\end{equation}
For a continuous function with zero exterior data in a bounded domain, \eqref{eq:tail-class} follows directly from continuity on the compact closure. The polynomial growth bounds used later also imply it on compact time intervals by dominated convergence.

\begin{Definition}[Viscosity solutions]\label{defn:vis}
Let $u:\R^n\times[-T,0]\to[0,\infty)$ satisfy \eqref{eq:tail-class} and be upper semicontinuous on $\overline\Omega\times[-T,0]$. It is a viscosity subsolution of
\begin{equation}\label{fpme}
\partial_tu=u^\gamma Lu\quad\text{in }\Omega_T
\end{equation}
if the following condition holds. Choose $(x_0,t_0)\in\Omega_T$, $r,\tau>0$ with $\overline{B_{2r}(x_0)}\subset\Omega$ and $t_0-\tau>-T$, and
\[
\varphi\in C^{2,1}_{x,t}\bigl(B_{2r}(x_0)\times(t_0-2\tau,t_0+2\tau)\bigr)
\]
such that
\[
\varphi(x_0,t_0)=u(x_0,t_0),\qquad \varphi\geq u\quad\text{on }B_r(x_0)\times(t_0-\tau,t_0].
\]
Define
\[
u_\varphi(x,t)=
\begin{cases}
\varphi(x,t),&(x,t)\in U:=B_r(x_0)\times(t_0-\tau,t_0+\tau),\\
u(x,t),&(x,t)\in \R^n\times[-T,0]\setminus U.
\end{cases}
\]
Then
\[
\partial_t\varphi(x_0,t_0) \leq u(x_0,t_0)^\gamma L\bigl(u_\varphi(\cdot,t_0)\bigr)(x_0).
\]
A function $u:\R^n\times[-T,0]\to[0,\infty)$ satisfying \eqref{eq:tail-class} and lower semicontinuous on $\overline\Omega\times[-T,0]$ is a viscosity supersolution of \eqref{fpme} if the corresponding condition holds with lower touching and the reversed inequality.

A function $u:\R^n\times[-T,0]\to[0,\infty)$ satisfying \eqref{eq:tail-class} and continuous on $\overline\Omega\times[-T,0]$ is a viscosity solution of \eqref{fpme} if it is both a viscosity subsolution and a viscosity supersolution.
\end{Definition}

Unless the operator class is explicitly enlarged, the extremal operators associated
with $\mathcal L_*$ are
\[
M^+f=\sup_{L\in\mathcal L_*}Lf, \qquad M^-f=\inf_{L\in\mathcal L_*}Lf.
\]
They are positively homogeneous, $M^-(-f)=-M^+f$, and
\[
M^-(f-g)\leq M^\pm f-M^\pm g\leq M^+(f-g).
\]
These viscosity difference inequalities follow from the usual doubling argument; see Theorem~5.9 of \cite{luis} and the parabolic framework in \cite{Chang2014,changstability}.

The operators in $\mathcal L_*$ also admit a directional representation. For $\theta\in\mathbb{S}^{n-1}$, 
\[
-(-\partial_{\theta\theta})^\sigma f(x)=c_{1,\sigma}\int_{\R}\delta f(x,\rho\theta)\frac{d\rho}{|\rho|^{1+2\sigma}},
\qquad Lf=\frac1{2c_{1,\sigma}}\int_{\mathbb S^{n-1}}\mu(\theta) \bigl(-(-\partial_{\theta\theta})^\sigma f\bigr)\,d\theta,
\]
where the constant $c_{1,\sigma}$ is given in \cite{RosOtonSerra2014Dirichlet}. 
Thus $M^+$ is obtained by integrating $\Lambda A^+-\lambda A^-$, and $M^-$ by integrating $\lambda A^+-\Lambda A^-$, with $A=-(-\partial_{\theta\theta})^\sigma f$ and the factor $(2c_{1,\sigma})^{-1}$. In particular, every operator in $\mathcal L_*$ annihilates $(x_n)_+^\sigma$ in the upper half-space. It follows that
\begin{equation}\label{eq:subtract-power}
M^\pm\bigl(f-a(x_n)_+^\sigma\bigr)=M^\pm f
\quad\text{in }\{x_n>0\},\qquad a\in\R.
\end{equation}

Definition~\ref{defn:vis} also applies to real-valued functions satisfying equations with a prescribed coefficient and $M^+$ or $M^-$. The nonnegativity restriction is needed only for the power $u^\gamma$. Elliptic inequalities use the spatial version of the same patching procedure.  

We will use the following well-known interior estimates.

\begin{Theorem}[Interior H\"older and Harnack estimates]\label{thm:interior-estimates}
Let $u:\R^n\times [-1,0]\to\R$ satisfy \eqref{eq:tail-class} with $I=[-1,0]$ and be continuous on $\overline{B_2}\times [-1,0]$. Suppose that, for some $F\geq0$,
\[
\partial_tu-M^-u\geq-F,\qquad
\partial_tu-M^+u\leq F
\quad\text{in }B_2\times(-1,0]
\]
in the viscosity sense. There exist $\beta\in(0,\min\{1,2\sigma\})$ and $C>0$, depending only on $n,\sigma,\lambda,\Lambda$, such that
\begin{equation}\label{eq:interior-holder}
\begin{aligned}
&\|u\|_{C_{x,t}^{\beta,\beta/(2\sigma)}(B_{1/2}\times[-1/2,0])}\leq C\bigg(\|u\|_{L^\infty(B_2\times [-1,0])}+F
+\sup_{t\in I}\int_{\R^n}\frac{|u(y,t)|}{1+|y|^{n+2\sigma}}\,dy\bigg).
\end{aligned}
\end{equation}
If, in addition, $u\geq0$ in $B_2\times [-1,0]$, then
\begin{equation}\label{eq:interior-harnack}
\begin{aligned}
\sup_{B_{1/2}\times[-3/4,-1/2]}u
&\leq C\bigg(\inf_{B_{1/2}\times[-1/4,0]}u+F+\sup_{t\in I}\int_{\R^n\setminus B_2}
\frac{u^-(y,t)}{1+|y|^{n+2\sigma}}\,dy\bigg).
\end{aligned}
\end{equation}
\end{Theorem}
\begin{proof}
The H\"older estimate \eqref{eq:interior-holder} follows from Theorem~5.1 of \cite{Chang2014}. The Harnack estimate \eqref{eq:interior-harnack} follows from \cite{chang2016, KimLee2013}\footnote{Although the Harnack result in \cite{chang2016} is stated for $\sigma\in(1/2,1)$ and the critical case $\sigma=1/2$, it allows for a critical drift term, and its proof can be readily adapted to the symmetric operators considered here for every $\sigma\in(0,1)$.}  with a truncation argument. Let 
\[
N=\sup_{t\in I}\int_{\mathbb R^n\setminus B_2}
\frac{u^-(y,t)}{1+|y|^{n+2\sigma}}\,dy
\] 
and $u_k=\min\{u^+,k\}$, where $k>\|u\|_{L^\infty(B_2\times I)}$. Then $u_k=u$ in $B_2\times I$. It follows from \eqref{eq:tail-class} that 
\begin{align*}
\sup_{L\in\mathcal L_*}|L(u-u_k)|
&\leq C\bigg(N+\sup_{t\in I}\int_{\R^n}
\frac{(u(y,t)-k)^+}{1+|y|^{n+2\sigma}}\,dy\bigg)=CN+o(1)\quad\mbox{as }k\to\infty
\end{align*}
uniformly on $B_{3/2}\times I$. Thus $u_k$ satisfies both extremal inequalities in $B_{3/2}\times(-1,0]$ with right-hand sides bounded by $F+CN+o(1)$. Applying the Harnack inequality in \cite{chang2016, KimLee2013} and letting $k\to\infty$ proves \eqref{eq:interior-harnack}. 
\end{proof}

\subsection{Comparison principles}
The first comparison principle will be used to compare solutions with the parabolic barriers. The sign condition is understood as an elliptic viscosity inequality at every fixed time.

\begin{Proposition}[Conditional comparison principle]\label{conditional}
Let $u,v:\R^n\times[-T,0]\to[0,\infty)$ satisfy \eqref{eq:tail-class} and be continuous on $\overline\Omega\times[-T,0]$. Suppose
\[
\partial_tu\leq u^\gamma Lu,\qquad
\partial_tv\geq v^\gamma Lv\quad\text{in }\Omega_T,
\]
in the viscosity sense, and
\[
u\leq v\quad\text{on }\Omega\times\{-T\},\qquad
u\leq v\quad\text{on }\Omega^c\times[-T,0].
\]
Exterior inequalities may be imposed almost everywhere in space, with the boundary values given by continuity. If $Lv(\cdot,t)\leq0$ in $\Omega$ for every $t\in(-T,0]$, then $u\leq v$ on $\overline\Omega\times[-T,0]$.
\end{Proposition}

\begin{Remark}\label{remcomparison}
The conclusion also holds if the sign assumption is replaced by $Lu(\cdot,t)\leq0$ in $\Omega$. The proof is given in Appendix~\ref{app:comparison}.
\end{Remark}

\begin{Proposition}[Comparison for degenerate extremal equations]\label{com for weight}
Let $D$ be a bounded open set and let $a\in C(\overline D)$ be positive in $D$, with $a^{1/(2\sigma)}$ locally Lipschitz in $D$. Let $u,v$ satisfy \eqref{eq:tail-class} and be continuous on $\overline D\times[-T,0]$. Suppose, for $M=M^-$ or $M=M^+$ and a constant $F\in\R$, that
\[
\partial_tu-a(x)Mu\leq F,\qquad
\partial_tv-a(x)Mv\geq F\quad\text{in }D\times(-T,0],
\]
and that $u\leq v$ initially in $D$ and on $D^c\times[-T,0]$. Then $u\leq v$ in $\overline D\times[-T,0]$. No sign condition on $u$ or $v$ is required.
\end{Proposition}

\begin{Corollary}[Half-space and distance weights]\label{com for dis}
Proposition~\ref{com for weight} applies with $a(x)=x_n^{\sigma\gamma}$ in every bounded open $D\subset\{x_n>0\}$, and with $a(x)=d(x)^{\sigma\gamma}$ in every bounded open $D\subset\Omega$.
\end{Corollary}
\begin{proof}
In a sufficiently small ball about any point of $D$, the relevant distance is bounded below by a positive constant. Both $x_n^{\gamma/2}$ and $d^{\gamma/2}$ are Lipschitz on that ball, since $d$ is Lipschitz. The hypotheses of the proposition are therefore satisfied.
\end{proof}

The proofs of these comparison principles in Appendix~\ref{app:comparison} use local doubling of variables and weighted $L^1$ convergence for the exterior integrals. 

\section{Elliptic barriers and two-sided boundary bounds}\label{sec:Lip}
\subsection{The nonlinear elliptic profile}
We first recall two auxiliary linear Dirichlet problems. The solution $h$ of
\begin{equation}\label{Lh=-1}
Lh=-1\quad\text{in }\Omega,\qquad h=0\quad\text{in }\Omega^c
\end{equation}
is continuous, positive in $\Omega$, and satisfies
\[
c d(x)^\sigma\le h(x)\leq C d(x)^\sigma\quad\text{in }\Omega.
\]
Existence, comparison, and the boundary bound follow from the linear Dirichlet theory in \cite{luis,ros2016boundary,RosOtonSerra2016Stable}. We next record the properties of the first eigenfunction used here and in the blow-up argument.

\begin{Lemma}[The first eigenfunction]\label{firsteigenfunction}
There is a positive first eigenpair $(\lambda_1,\psi)$ for $-L$ with zero exterior condition, normalized by $\|\psi\|_{L^2(\Omega)}=1$. It satisfies
\begin{equation}\label{eigenvalue}
-L\psi=\lambda_1\psi\quad\text{in }\Omega,\qquad \psi=0\quad\text{in }\Omega^c,
\end{equation}
and
\[
c d^\sigma\leq\psi\leq C d^\sigma,\quad
\psi\in C^\sigma(\R^n)\cap C^{3\sigma-\epsilon}_{loc}(\Omega),\quad
\psi/d^\sigma\in C^\beta(\overline\Omega)
\]
for every $0<\beta<\sigma$ and every small $\epsilon>0$ whenever $3\sigma-\epsilon$ is not an integer. The eigenvalue is bounded above and below by positive constants depending only on $n,\sigma,\lambda,\Lambda$ and $\Omega$. The constants in the displayed estimates have the same dependence, and the H\"older norm also depends on $\beta$.
\end{Lemma}

The existence and boundary behavior of the first eigenfunction follow from the spectral and boundary theory in Section~5 of \cite{BFV2018}. The quotient estimate follows by applying Theorem~1.2 of \cite{RosOtonSerra2016Stable} to \eqref{eigenvalue}; boundedness of the normalized eigenfunction follows from the same elliptic theory. The eigenvalue bounds also follow directly by comparing the quadratic form of $-L$ with that of the fractional Laplacian. We refer to these works for the proofs.

\begin{Lemma}[Elliptic comparison principle]\label{elliptic comparison}
Let $v,w:\R^n\to[0,\infty)$ be continuous on $\overline\Omega$ and integrable with weight $(1+|x|^{n+2\sigma})^{-1}$ in $\R^n$. Suppose $v\not\equiv0$ in $\Omega$,
\[
Lv+\frac{1-\gamma}{\gamma}v^{1-\gamma}\leq0
\leq Lw+\frac{1-\gamma}{\gamma}w^{1-\gamma}
\quad\text{in }\Omega,
\]
in the viscosity sense, $v\geq w$ on $\Omega^c$, and $v\leq c d^\sigma$ in $\Omega$ for some $c>0$. Then $v\geq w$ in $\Omega$.
\end{Lemma}

\begin{proof}
\textbf{Step 1. }
We first give the strong maximum principle argument needed below. Suppose $f\geq0$ on $\R^n$, $f\not\equiv0$ in $\Omega$, and $M^-f\leq0$ in $\Omega$. If $f(x_0)=0$ at an interior point, the zero function touches from below. Patching by zero on $B_\rho(x_0)$ gives
\[
0\geq\inf_\mu\int_{|z|>\rho}
\frac{f(x_0+z)+f(x_0-z)}2
\frac{\mu(z/|z|)}{|z|^{n+2\sigma}}\,dz
\geq\lambda\int_{|z|>\rho}
\frac{f(x_0+z)+f(x_0-z)}{2|z|^{n+2\sigma}}\,dz.
\]
Thus $f=0$ almost everywhere, and continuity in $\Omega$ contradicts nontriviality there. Hence $f>0$ in $\Omega$.

Translating and scaling the barrier in Lemma~3.4 of
\cite{ros2016boundary} in uniform interior tangent balls, and
taking their maximum, gives a bounded continuous nonnegative
function $\varphi_2$, vanishing in $\Omega^c$, such that
\[
M^-\varphi_2\geq c>0,\qquad \varphi_2\geq c\,d^\sigma
\quad\text{in }\Omega_\rho:=\{x\in\Omega:d(x)<\rho\}
\]
for some $\rho>0$. Since
$\min_{\Omega\setminus\Omega_\rho}f>0$, we may choose $a>0$
sufficiently small that
$a\varphi_2\leq f$ in $\mathbb R^n\setminus\Omega_\rho$.
The comparison principle then gives
\[
f\geq a\varphi_2\geq ac\,d^\sigma
\quad\text{in }\Omega_\rho.
\]
Together with positivity on the remaining compact interior set,
this yields $f\geq c_1d^\sigma$ in $\Omega$, where $c_1>0$
depends on $f$.

\textbf{Step 2. }
Applying Step~1 to $v$ gives $v\geq c_vd^\sigma$. Set
\[
C_*:=\frac{1-\gamma}{\gamma}\|w\|_{L^\infty(\Omega)}^{1-\gamma}.
\]
This local norm is finite by continuity on $\overline\Omega$. The viscosity difference inequality for the fixed linear operator gives
\[
L(w-v)\geq-C_*\quad\text{in }\Omega.
\]
Since $w-v\leq0$ in the exterior, comparison with $C_*h$, where $h$ is the solution of \eqref{Lh=-1}, yields
\[
w\leq v+C_*h\leq C_w d^\sigma\quad\text{in }\Omega.
\]
Consequently $v\geq t w$ everywhere for all sufficiently small $t>0$.

\textbf{Step 3.}
If the conclusion fails, define
\[
t^*:=\sup\{t\in[0,1]:v\geq tw\text{ on }\R^n\}.
\]
By Step~2, $t^*>0$; the existence of a point where $v<w$ implies $t^*<1$. The defining inequality also holds at $t=t^*$. Set $f=v-t^*w\geq0$. By the linear viscosity difference inequality,
\begin{align*}
M^-f&\leq Lf\leq\frac{1-\gamma}{\gamma}
\bigl(t^*w^{1-\gamma}-v^{1-\gamma}\bigr)\leq-\frac{1-\gamma}{\gamma}
\bigl(1-(t^*)^\gamma\bigr)v^{1-\gamma}<0
\quad\text{in }\Omega,
\end{align*}
where we use $t^*w\leq v$ in the third inequality. Since $f\ge 0$ globally, we have $f>0$ in $\Omega$, and the boundary part of Step~1 gives $f\geq c_fd^\sigma$.

Choose $0<\epsilon<\min\{1-t^*,c_f/c\}$. Inside $\Omega$, $f\geq\epsilon v$ by the assumed upper bound for $v$. Outside, $v\geq w\geq0$ gives $f\geq(1-t^*)v\geq\epsilon v$. Therefore
\[
v\geq\frac{t^*}{1-\epsilon}w\quad\text{on }\R^n,
\]
where $t^*<t^*/(1-\epsilon)<1$. This contradicts the definition of $t^*$.
\end{proof}

\begin{Theorem}[The elliptic profile]\label{Dirichlet}
For every $L\in\mathcal L_*$, the problem
\begin{equation}\label{eq:ellipticnonlinear}
L\Psi+\frac{1-\gamma}{\gamma}\Psi^{1-\gamma}=0
\quad\text{in }\Omega,\qquad \Psi=0\quad\text{in }\Omega^c
\end{equation}
admits a unique continuous viscosity solution that is positive in $\Omega$. Moreover,
\begin{equation}\label{bounary}
c_0d(x)^\sigma\leq\Psi(x)\leq C_0d(x)^\sigma\quad\text{in }\Omega,
\end{equation}
where $c_0,C_0>0$ depend only on $n,\sigma,\gamma,\lambda,\Lambda$ and $\Omega$.
\end{Theorem}

\begin{proof}
Theorem~6.2 of \cite{BFV2018} gives a nontrivial weak dual solution, and Theorem~6.3 there gives boundedness. The Green representation of that solution identifies it with the linear weak Dirichlet solution with right-hand side $-\frac{1-\gamma}{\gamma}\Psi^{1-\gamma}$. This right-hand side is bounded, so the linear boundary regularity in \cite{RosOtonSerra2016Stable} gives H\"older continuity up to the boundary. The right-hand side is consequently H\"older continuous as well. Hence, $\Psi\in C^{2\sigma+\alpha}_{loc}(\Omega)$ for some $\alpha>0$. Therefore, it is a viscosity solution. 

We prove the uniform bounds and uniqueness in the viscosity class. Choose $A>0$ so that
\[
A^\gamma\geq\frac{1-\gamma}{\gamma}\|h\|_\infty^{1-\gamma}.
\]
Then
\[
L(Ah)+\frac{1-\gamma}{\gamma}(Ah)^{1-\gamma}
\leq-A+\frac{1-\gamma}{\gamma}A^{1-\gamma}\|h\|_\infty^{1-\gamma}\leq0.
\]
It follows from Lemma~\ref{elliptic comparison} that $\Psi\leq Ah\leq C_0d^\sigma$.

Choose $B>0$ sufficiently small that
\[
\lambda_1(B\|\psi\|_\infty)^\gamma\leq\frac{1-\gamma}{\gamma}.
\]
The eigenfunction equation gives
\[
L(B\psi)+\frac{1-\gamma}{\gamma}(B\psi)^{1-\gamma}
=(B\psi)^{1-\gamma}
\left[\frac{1-\gamma}{\gamma}-\lambda_1(B\psi)^\gamma\right]\geq0.
\]
Since every positive solution satisfies the required upper bound, comparison with $B\psi$ gives $\Psi\geq B\psi\geq c_0d^\sigma$. Finally, applying Lemma~\ref{elliptic comparison} with the two positive solutions in either order proves uniqueness. All constants used above have the stated dependence.
\end{proof}

\subsection{Parabolic barriers and two-sided boundary bounds}
The elliptic profile $\Psi$ from Theorem~\ref{Dirichlet} gives
solutions with separated space and time variables. By choosing
suitable time shifts and applying the comparison principles,
we show that $u(\cdot,t)$ remains bounded above and below by
positive multiples of $d^\sigma$ throughout $[-T,0]$.

\begin{Theorem}\label{thm1}
Under the hypotheses of Theorem~\ref{mainthm}, there are $C_1,C_2>0$, depending only on $n,\sigma,\gamma,\lambda,\Lambda,c_1,c_2,T$ and $\Omega$, such that
\[
C_1d(x)^\sigma\leq u(x,t)\leq C_2d(x)^\sigma
\quad\text{on }\overline\Omega\times[-T,0].
\]
\end{Theorem}
\begin{proof}
Let $\Psi$ be the solution of \eqref{eq:ellipticnonlinear}. For $\tau>0$, define
\[
U_\tau(x,t)=(1-\gamma)^{-1/\gamma}
(t+T+\tau)^{-1/\gamma}\Psi(x).
\]
Writing the scalar time factor as $b(t)$, we have
\[
b'=-\frac{1-\gamma}{\gamma}b^{1+\gamma},\qquad
\partial_tU_\tau=b'\Psi
=b^{1+\gamma}\Psi^\gamma L\Psi
=U_\tau^\gamma LU_\tau.
\]
The same identities hold in the viscosity sense: divide the touching test function by $b(t)$, use the same patch cylinder for both functions, and apply the linearity of $L$. Moreover $LU_\tau<0$ in $\Omega$.

Take
\[
\tau_+=\frac1{1-\gamma}\left(\frac{c_0}{c_2}\right)^\gamma,
\qquad
\tau_-=\frac1{1-\gamma}\left(\frac{C_0}{c_1}\right)^\gamma.
\]
The initial bounds imply $U_{\tau_-}(x,-T)\leq u_0(x)\leq U_{\tau_+}(x,-T)$. The exterior ordering holds because all three functions vanish there. Proposition~\ref{conditional} gives the upper comparison with $U_{\tau_+}$, and Remark~\ref{remcomparison} gives the lower comparison with $U_{\tau_-}$. Hence
\[
U_{\tau_-}\leq u\leq U_{\tau_+}\quad\text{on }\overline\Omega\times[-T,0].
\]
Using \eqref{bounary}, we may take
\[
C_1=c_0(1-\gamma)^{-1/\gamma}(T+\tau_-)^{-1/\gamma},
\qquad
C_2=C_0(1-\gamma)^{-1/\gamma}\tau_+^{-1/\gamma}.
\]
This proves the estimate with the stated time dependence.
\end{proof}

The preceding theorem identifies the degeneracy as $u^\gamma\asymp d^{\sigma\gamma}$. To study continuity of the quotient, we will subtract multiples of $\psi$ and rescale. The resulting errors may change sign and satisfy degenerate extremal inequalities. We therefore next establish half-space estimates for these inequalities, followed by boundary estimates in curved domains.

\section{Half-space estimates, a Hopf lemma, and a Liouville theorem}\label{sec:liouville}
In this section, we consider equations in the upper half-space and use the cylinders $Q_r^+$. The constants depend only on $n,\sigma,\gamma,\lambda,\Lambda$, unless additional dependence is specified. The boundary barriers and the oscillation argument are related to those in \cite{ros2016boundary,LeeYun2025}.

\subsection{Boundary growth estimates}
We first establish an upper bound for $|u|/x_n^\sigma$ in a smaller
cylinder. The estimate applies to functions of either sign satisfying
the degenerate extremal inequalities, and is controlled by the local
supremum, the weighted tail, and the size of the right-hand side.

\begin{Lemma}\label{nonhmlip}
Let $u:\R^n\times[-1,0]\to\R$ satisfy \eqref{eq:tail-class}, be continuous on $\overline{Q_1^+}$, and vanish on $\{x_n\leq0\}\times[-1,0]$. Suppose, for $F\geq0$, that
\[
\partial_tu-x_n^{\sigma\gamma}M^+u\leq F,\qquad
\partial_tu-x_n^{\sigma\gamma}M^-u\geq-F
\quad\text{in }Q_1^+.
\]
Set
\[
N=\|u\|_{L^\infty(B_1\times[-1,0])}
+\sup_{t\in[-1,0]}\int_{|y|>1}\frac{|u(y,t)|}{|y|^{n+2\sigma}}\,dy.
\]
Then
\[
|u(x,t)|\leq C(N+F)x_n^\sigma\quad\text{on }\overline{Q_{1/2}^+}.
\]
\end{Lemma}

\begin{proof}
\textbf{Step 1. }
Take a smooth cutoff $\chi$ supported in $B_1$, with $0\leq\chi\leq1$ and $\chi=1$ on $B_{3/4}$, and set $w=\chi u$. For $x\in B_{5/8}$, $w(x,t)=u(x,t)$ and, uniformly in $L\in\mathcal L_*$,
\begin{align*}
|L(w-u)(x,t)|
&\leq C\int_{|y|>3/4}\frac{|u(y,t)|}{|x-y|^{n+2\sigma}}\,dy\leq CN.
\end{align*}
Hence, in the viscosity sense,
\[
\begin{cases}
\partial_t w-x_n^{\sigma\gamma}M^+w\leq C(F+N),\\
\partial_t w-x_n^{\sigma\gamma}M^-w\geq -C(F+N),
\end{cases}
\qquad\text{in }B_{5/8}^+\times(-1,0].
\]

If $N+F=0$, the result is immediate. Otherwise, divide $w$ by $C(N+F)$ and retain the notation $w$. Then $|w|\leq1$ globally, and the right-hand sides are bounded by $1$. Set $\widetilde w=(t+1)w$, with $\widetilde w(\cdot,-1)=0$. For an upper test function $\widetilde\varphi$, set $\varphi=\widetilde\varphi/(t+1)$ and use the same patch cylinder. Positive homogeneity gives
\[
\partial_t\widetilde\varphi-x_n^{\sigma\gamma}M^+\widetilde w_{\widetilde\varphi}
\leq(t+1)+w\leq2.
\]
All terms are evaluated at the contact point. 

\textbf{Step 2. }
Let
\[
\Phi_0(x):=2(x_n)_+^\sigma-(x_n)_+^{\sigma(2-\gamma)},\quad \Phi(x)=2\min\{(x_n)_+,1\}^\sigma-
\min\{(x_n)_+,1\}^{\sigma(2-\gamma)}.
\]
Then $\Phi$ satisfies
\[
\min\{(x_n)_+,1\}^\sigma\leq\Phi(x)\leq2(x_n)_+^\sigma.
\]
The half-space power formula, Lemmas~2.2--2.3 of \cite{ros2016boundary}, gives
\[
L((x_n)_+^\sigma)=0,\qquad
L((x_n)_+^{\sigma(2-\gamma)})=c_Lx_n^{-\sigma\gamma},
\qquad 0<c_*\leq c_L\leq C_*.
\]
For $0<x_n<1/4$,
\[
\delta(\Phi-\Phi_0)(x,z)=0\quad\text{if }|z|<3/4,
\qquad
|\delta(\Phi-\Phi_0)(x,z)|
\leq C\bigl(1+|z|^{\sigma(2-\gamma)}\bigr).
\]
Hence, uniformly in $L\in\mathcal L_*$,
\begin{align*}
|L(\Phi-\Phi_0)(x)|
&\leq C\int_{|z|\geq3/4}
\frac{1+|z|^{\sigma(2-\gamma)}}{|z|^{n+2\sigma}}\,dz\leq C\int_{3/4}^{\infty}
\bigl(r^{-1-2\sigma}+r^{-1-\sigma\gamma}\bigr)\,dr
\leq C.
\end{align*}
Therefore,
\[
M^+\Phi
\leq M^+\Phi_0+\sup_{L\in\mathcal L_*}|L(\Phi-\Phi_0)|
\leq-c_*x_n^{-\sigma\gamma}+C.
\]
Choose $\epsilon\in(0,1/32)$ such that
$C\epsilon^{\sigma\gamma}\leq c_*/2$. Then
\[
-x_n^{\sigma\gamma}M^+\Phi
\geq c_*-Cx_n^{\sigma\gamma}
\geq c_*/2
\qquad\text{for }0<x_n<\epsilon.
\]

\textbf{Step 3. }
Fix $x'\in\overline{B_{1/2}}\cap\{x_n=0\}$ and choose $\rho\in C^\infty(\R^n)$ such that $0\leq\rho\leq1$, $\rho=0$ in $B_{1/16}(x')$, $\rho=1$ in $\R^n\setminus B_{1/8}(x')$, $\sup_{L\in\mathcal L_*}\|L\rho\|_{L^\infty(\R^n)}\leq C$. Set
\[
D:=B_{1/8}(x')\cap\{0<x_n<\epsilon\}\subset B_{5/8}^+,
\]
and choose $K>0$ such that
\[
K\epsilon^\sigma\geq1,\qquad
\frac{Kc_*}{2}-C\epsilon^{\sigma\gamma}\geq2.
\]
In $D\times(-1,0]$,
\begin{align*}
\partial_t(K\Phi+\rho)-x_n^{\sigma\gamma}M^+(K\Phi+\rho)
&\geq-Kx_n^{\sigma\gamma}M^+\Phi-x_n^{\sigma\gamma}M^+\rho\geq\frac{Kc_*}{2}-C\epsilon^{\sigma\gamma}\geq2.
\end{align*}
Moreover,
\[
K\Phi(x)+\rho(x)\geq
\begin{cases}
1,&x\notin B_{1/8}(x'),\\
K\epsilon^\sigma\geq1,&x_n\geq\epsilon,\\
0,&x_n\leq0,
\end{cases}
\]
and hence
\[
\begin{cases}
\widetilde w\leq K\Phi+\rho
&\text{on }D^c\times[-1,0],\\
\widetilde w(\cdot,-1)=0\leq K\Phi+\rho
&\text{in }\R^n.
\end{cases}
\]
Proposition~\ref{com for weight} gives
\[
\widetilde w\leq K\Phi+\rho\quad\text{in }D\times(-1,0].
\]
For $x\in\overline{B_{1/2}}$ with $0<x_n<\epsilon$, set $x'=(x_1,\ldots,x_{n-1},0)$. Then $|x-x'|=x_n<\epsilon<1/16$, $x\in D$, $\rho(x)=0$, and
\[
(t+1)w(x,t)\leq K\Phi(x)\leq2Kx_n^\sigma
\quad\text{for }t\in[-1,0].
\]
Since $t+1\geq1-2^{-\sigma(2-\gamma)}>0$ on $\overline{Q_{1/2}^+}$, 
we obtain, on $\overline{Q_{1/2}^+}$,
\[
w(x,t)\leq
\begin{cases}
\displaystyle\frac{2K}{1-2^{-\sigma(2-\gamma)}}x_n^\sigma,
&0<x_n<\epsilon,\\[6pt]
1\leq\epsilon^{-\sigma}x_n^\sigma,&x_n\geq\epsilon,\\
0,&x_n=0.
\end{cases}
\]
Using $M^+(-w)=-M^-w$, the same argument for $-w$ yields
\[
|w(x,t)|\leq Cx_n^\sigma
\quad\text{on }\overline{Q_{1/2}^+}.
\]
Restoring the normalization and using $\chi=1$ on $B_{1/2}$ gives
\[
|u(x,t)|\leq C(N+F)x_n^\sigma
\quad\text{on }\overline{Q_{1/2}^+}.
\]
\end{proof}

\subsection{A Hopf lemma with tail control}
We next establish a quantitative Hopf lemma for functions that are
nonnegative in $Q_1^+$. It bounds $u/x_n^\sigma$ from below in terms
of an interior value at an earlier time, with error terms involving
$F$ and the negative tail outside $B_1$. This estimate will be used
in the oscillation argument in Proposition~\ref{flat}.

\begin{Lemma}[Hopf principle]\label{tailhpf}
Let $u:\R^n\times[-1,0]\to\R$ satisfy \eqref{eq:tail-class} and be continuous on $\overline{Q_1^+}$. Suppose, for some $F\geq0$, that
\[
\partial_tu-x_n^{\sigma\gamma}M^-u\geq-F,\qquad
\partial_tu-x_n^{\sigma\gamma}M^+u\leq F\quad\text{in }Q_1^+,
\]
$u\geq0$ on $\overline{Q_1^+}$, and $u=0$ on $\{x_n\leq0\}\times[-1,0]$. Then there are structural constants $c_h,c_H>0$ such that
\begin{equation}\label{hpfineq}
\begin{aligned}
u(x,t)\geq\Bigg[&c_hu\left(\frac{e_n}{2},-\left(\frac23\right)^{\sigma(2-\gamma)}\right)-c_H\left(F+\sup_{\tau\in[-1,0]}\int_{|y|>1}\frac{u^-(y,\tau)}{|y|^{n+2\sigma}}\,dy\right)\Bigg]x_n^\sigma
\end{aligned}
\end{equation}
on $\overline{Q_{1/2}^+}$. 
\end{Lemma}

\begin{proof}
We first assume $F=0$. Following the barrier construction in \cite{WangChen2022Hopf}, set
\begin{align*}
A:=u\left(\frac{e_n}{2},-\left(\frac23\right)^{\sigma(2-\gamma)}\right)\geq0,\qquad
N:=\sup_{t\in[-1,0]}\int_{|y|>1}
\frac{u^-(y,t)}{|y|^{n+2\sigma}}\,dy<\infty.
\end{align*}
If $A=0$, then $u\geq0\geq-c_HNx_n^\sigma$ on $\overline{Q_{1/2}^+}$. We may therefore assume $A>0$.
\medskip

\textbf{Step 1. }
Fix $r=\frac18$,  $D=B_{1/16}(e_n/2)$, $\xi_D\in C_c^\infty(D)$, $0\leq\xi_D\leq1$, and $\xi_D\not\equiv 0$. For $x'\in\overline{B_{1/2}}\cap\{x_n=0\}$, $B_r(x'+re_n)\subset B_{3/4}^+$, $B_r(x'+re_n)\cap D=\emptyset$ and $\overline{B_{r/2}(x'+re_n)}\subset\{x_n\geq r/2\}$. Set
\begin{gather*}
\tau_0:=-\frac12\left[\left(\frac23\right)^{\sigma(2-\gamma)}
+2^{-\sigma(2-\gamma)}\right],\quad
-\left(\frac23\right)^{\sigma(2-\gamma)}
<\tau_0<-2^{-\sigma(2-\gamma)}.
\end{gather*}
For $x\in B_{3/4}$ and $L\in\mathcal L_*$,
\begin{align*}
0\leq L(u^+-u)(x,t)
&\leq\Lambda\int_{|y|>1}
\frac{u^-(y,t)}{|y-x|^{n+2\sigma}}\,dy
\leq CN.
\end{align*}
Thus
\[
\begin{cases}
\partial_tu^+-x_n^{\sigma\gamma}M^+u^+\leq0,\\
\partial_tu^+-x_n^{\sigma\gamma}M^-u^+\geq-CN,
\end{cases}
\qquad\text{in }B_{3/4}^+\times(-1,0],
\]
and $(r/4)^{\sigma\gamma}\leq x_n^{\sigma\gamma}\leq1$ in $B_{3/4}\cap\{x_n>r/4\}$. By these bounds, the coefficient $x_n^{\sigma\gamma}$ can be absorbed into the ellipticity constants. The Harnack estimate \eqref{eq:interior-harnack} in Theorem~\ref{thm:interior-estimates}, applied to $u^+$ along a fixed finite chain of cylinders in this set, gives, uniformly in $x'$,
\begin{equation}\label{eq:hopf-interior}
\begin{aligned}
\inf_{D\times[\tau_0,0]}u&\geq cA-CN,\\
\inf_{\overline{B_{r/2}(x'+re_n)}\times[\tau_0,0]}u
&\geq cA-CN,\\
\inf_{(\overline{B_{1/2}^+}\cap\{x_n\geq r/2\})
\times[-2^{-\sigma(2-\gamma)},0]}u
&\geq cA-CN.
\end{aligned}
\end{equation}

\medskip

\textbf{Step 2. }
Fix $x'$ as above and set
\[
A_r:=B_r(x'+re_n)\setminus\overline{B_{r/2}(x'+re_n)}.
\]
Translating and scaling the barrier constructed in the proof of Lemma~3.4 in \cite{ros2016boundary}, we obtain a function $\varphi_2\in C(\R^n)$ satisfying
\begin{align*}
M^-\varphi_2&\geq c_0>0\quad\text{in }A_r,\\
\varphi_2&=0\quad\text{in }B_r(x'+re_n)^c,\\
\varphi_2(x)&\leq C(x_n)_+^\sigma\quad\mbox{for }x\in\R^n,\\
\varphi_2(x'+he_n)&\geq ch^\sigma\quad\mbox{for }0<h<r/2.
\end{align*}
Choose $\eta\in C^\infty(\R)$ such that $0\leq\eta\leq1$, $0\leq\eta'\leq C$, $\eta(t)=0$ when $t\leq\tau_0$ and $\eta(t)=1$ when $t\geq-2^{-\sigma(2-\gamma)}$. Define
\[
v(x,t)=\xi_D(x)u(x,t)+\epsilon\eta(t)\varphi_2(x)
-\chi_{\R^n\setminus B_1}(x)u^-(x,t),
\qquad t\in[\tau_0,0].
\]
By \eqref{eq:tail-class},
\begin{align*}
\int_{\R^n}\frac{|v(y,t)-v(y,s)|}{1+|y|^{n+2\sigma}}\,dy
&\leq2\int_{\R^n}
\frac{|u(y,t)-u(y,s)|}{1+|y|^{n+2\sigma}}\,dy+C\epsilon|\eta(t)-\eta(s)|\\
&\rightarrow0\qquad\mbox{as }s\to t.
\end{align*}

\medskip

\textbf{Step 3. }
Let $\psi$ be an upper test function at $(x_0,t_0)\in A_r\times(\tau_0,0]$, with $\overline{B_\ell(x_0)}\subset A_r$,  $\tau_0<t_0-\ell$, $\psi(x_0,t_0)=v(x_0,t_0)$, 
$\psi\geq v$  in $B_\ell(x_0)\times(t_0-\ell,t_0]$.
Let $v_\psi$ be the patch from Definition~\ref{defn:vis}, with spatial radius $\ell$. Since
\[
v=\epsilon\eta\varphi_2\quad\text{in }A_r\times[\tau_0,0],
\]
if $\eta(t_0)>0$, the spatial test
\[
\varphi(y):=\frac{\psi(y,t_0)}{\epsilon\eta(t_0)}
\]
satisfies
\[
\varphi(x_0)=\varphi_2(x_0),\qquad
\varphi\geq\varphi_2\quad\text{in }B_\ell(x_0),\qquad
M^-(\varphi_2)_\varphi(x_0)\geq c_0,
\]
where $(\varphi_2)_\varphi$ is patched in $B_\ell(x_0)$. If $\eta(t_0)=0$, then
\[
\eta'(t_0)=0,\qquad \psi(x_0,t_0)=0,\qquad
\delta\psi(x_0,z,t_0)\geq0\quad\mbox{for }|z|<\ell.
\]
Moreover,
\[
|y-x_0|\geq\frac14|y|\quad\mbox{for }|y|>1\quad\mbox{and}\quad
c\leq\int_D\frac{\xi_D(y)}{|y-x_0|^{n+2\sigma}}\,dy\leq C.
\]
In either case, \eqref{eq:hopf-interior} gives
\begin{align*}
M^-v_\psi(x_0,t_0)
&\geq\epsilon\eta(t_0)c_0
+\lambda\int_D\frac{\xi_D(y)u(y,t_0)}{|y-x_0|^{n+2\sigma}}\,dy-\Lambda\int_{|y|>1}
\frac{u^-(y,t_0)}{|y-x_0|^{n+2\sigma}}\,dy\geq cA-CN.
\end{align*}
The backward contact yields
\begin{align*}
\partial_t\psi(x_0,t_0)
&\leq\epsilon\eta'(t_0)\varphi_2(x_0),\\
x_{0,n}^{-\sigma\gamma}\partial_t\psi(x_0,t_0)
&\leq C\epsilon x_{0,n}^{\sigma(1-\gamma)}
\leq C\epsilon.
\end{align*}
Using common constants $c,C$ in these bounds and \eqref{eq:hopf-interior}, choose
\[
K\geq C/c,\qquad
cA-CN\geq c(A-KN).
\]
If $A\leq KN$, then
\[
u\geq0\geq c_h(A-KN)x_n^\sigma
\quad\text{on }\overline{Q_{1/2}^+}.
\]
For $A>KN$, set
\[
\epsilon=c'(A-KN),\qquad 0<c'\leq\frac{c}{2C}.
\]
Then
\begin{align*}
\partial_t\psi(x_0,t_0)
-x_{0,n}^{\sigma\gamma}M^-v_\psi(x_0,t_0)
&\leq x_{0,n}^{\sigma\gamma}
\bigl[C\epsilon-c(A-KN)\bigr]\\
&\leq-\frac c2x_{0,n}^{\sigma\gamma}(A-KN)\\
&\leq0.
\end{align*}
Hence
\[
\partial_tv-x_n^{\sigma\gamma}M^-v\leq0
\quad\text{in }A_r\times(\tau_0,0].
\]

\medskip

\textbf{Step 4. }
At $t=\tau_0$,
\[
v(x,\tau_0)=\xi_D(x)u(x,\tau_0)
-\chi_{\R^n\setminus B_1}(x)u^-(x,\tau_0)
\leq u(x,\tau_0)\quad(x\in\R^n).
\]
For $t\in[\tau_0,0]$,
\[
v(x,t)=
\begin{cases}
\xi_D(x)u(x,t)\leq u(x,t),
&x\in B_1\setminus B_r(x'+re_n),\\
-u^-(x,t)\leq u(x,t),
&x\in\R^n\setminus B_1.
\end{cases}
\]
On $\overline{B_{r/2}(x'+re_n)}\times[\tau_0,0]$,
\[
v=\epsilon\eta\varphi_2
\leq Cc'(A-KN)\leq c(A-KN)\leq u,
\]
after decreasing $c'$ if necessary. Proposition~\ref{com for weight} therefore gives
\[
u\geq v\quad\text{on }\overline{A_r}\times[\tau_0,0].
\]
For $0<h<r/2$,
\[
|x'+he_n-(x'+re_n)|=r-h\in(r/2,r),
\]
and hence
\begin{gather*}
u(x'+he_n,t)\geq\epsilon\varphi_2(x'+he_n)
\geq cc'(A-KN)h^\sigma \quad\mbox{for }-2^{-\sigma(2-\gamma)}\leq t\leq0.
\end{gather*}
For $x\in\overline{B_{1/2}^+}$ with $x_n\geq r/2$, \eqref{eq:hopf-interior} gives
\[
u(x,t)\geq c(A-KN)\geq c(A-KN)x_n^\sigma
\quad\bigl(-2^{-\sigma(2-\gamma)}\leq t\leq0\bigr).
\]
Taking $x'=(x_1,\ldots,x_{n-1},0)$ when $0<x_n<r/2$, and using continuity, we obtain
\[
u(x,t)\geq(c_hA-c_HN)x_n^\sigma
\quad\text{on }\overline{Q_{1/2}^+},
\qquad c_H=Kc_h,
\]
for a sufficiently small structural constant $c_h>0$.

\medskip

\textbf{Step 5. } Now let us consider the inhomogeneous case. Let $F>0$, and define $A,N$ as above. Set
\[
P(x):=(x_n)_+^\sigma-(x_n)_+^{\sigma(2-\gamma)}.
\]
The power identities in the proof of Lemma~\ref{nonhmlip} give, for $x_n>0$,
\[
LP(x)=-c_Lx_n^{-\sigma\gamma},\qquad
0<c_*\leq c_L\leq C_*\quad(L\in\mathcal L_*),
\]
and hence
\[
x_n^{\sigma\gamma}M^+P\leq-c_*,\qquad
x_n^{\sigma\gamma}M^-P\geq-C_*.
\]
Moreover, $\sigma<\sigma(2-\gamma)<2\sigma$, $0\leq P(x)\leq x_n^\sigma$ for $x\in B_1^+$, $P=0$ on $\{x_n\leq0\}$, and
\[
\int_{\R^n}\frac{|P(y)|}{1+|y|^{n+2\sigma}}\,dy\leq C.
\]
Define
\[
\widetilde u(x,t):=u(x,t)+\frac{F}{c_*}P(x).
\]
Since $P$ is independent of $t$, $\widetilde u$ satisfies \eqref{eq:tail-class}. Also, $\widetilde u\geq0$ on $\overline{Q_1^+}$, $\widetilde u=0$ on $\{x_n\leq0\}\times[-1,0]$, 
\[
\widetilde u\left(\frac{e_n}{2},-\left(\frac23\right)^{\sigma(2-\gamma)}\right)\geq A
\quad\mbox{and}\quad
\sup_{t\in[-1,0]}\int_{|y|>1}
\frac{\widetilde u^-(y,t)}{|y|^{n+2\sigma}}\,dy\leq N+CF.
\]
In the viscosity sense in $Q_1^+$,
\begin{align*}
\partial_t\widetilde u-x_n^{\sigma\gamma}M^-\widetilde u
&\geq-F-\frac{F}{c_*}x_n^{\sigma\gamma}M^+P\geq0,\\
\partial_t\widetilde u-x_n^{\sigma\gamma}M^+\widetilde u
&\leq F-\frac{F}{c_*}x_n^{\sigma\gamma}M^-P
\leq\left(1+\frac{C_*}{c_*}\right)F.
\end{align*}
As in Step~1,
\[
\begin{cases}
\partial_t\widetilde u^+-x_n^{\sigma\gamma}M^+\widetilde u^+\leq CF,\\
\partial_t\widetilde u^+-x_n^{\sigma\gamma}M^-\widetilde u^+\geq-C(N+F),
\end{cases}
\qquad\text{in }B_{3/4}^+\times(-1,0].
\]
Applying \eqref{eq:interior-harnack} in Theorem~\ref{thm:interior-estimates} to $\widetilde u^+$ along the same chain of cylinders yields
\[
\begin{aligned}
\inf_{D\times[\tau_0,0]}\widetilde u&\geq cA-C(N+F),\\
\inf_{\overline{B_{r/2}(x'+re_n)}\times[\tau_0,0]}\widetilde u
&\geq cA-C(N+F),\\
\inf_{(\overline{B_{1/2}^+}\cap\{x_n\geq r/2\})
\times[-2^{-\sigma(2-\gamma)},0]}\widetilde u
&\geq cA-C(N+F).
\end{aligned}
\]
Choose $K\geq C/c$. If $A\leq K(N+F)$, then \eqref{hpfineq} follows from $u\geq0$, provided $c_H\geq Kc_h$. Otherwise, use the same $\varphi_2,\eta$ as in Step~2 and set
\[
\epsilon=c'\bigl[A-K(N+F)\bigr],\qquad
v=\xi_D\widetilde u+\epsilon\eta\varphi_2
-\chi_{\R^n\setminus B_1}\widetilde u^-.
\]
For an upper test $\psi$ and its patch $v_\psi$ as in Step~3,
\[
M^-v_\psi(x_0,t_0)\geq cA-C(N+F)
\geq c\bigl[A-K(N+F)\bigr],\qquad
x_{0,n}^{-\sigma\gamma}\partial_t\psi(x_0,t_0)\leq C\epsilon.
\]
Using common constants $c,C$ above and choosing $c'>0$ sufficiently small gives
\[
\partial_tv-x_n^{\sigma\gamma}M^-v
\leq-\frac c2x_n^{\sigma\gamma}\bigl[A-K(N+F)\bigr]\leq0
\quad\text{in }A_r\times(\tau_0,0],
\]
\[
v\leq\widetilde u
\quad\text{on }\bigl(\R^n\times\{\tau_0\}\bigr)
\cup\bigl(A_r^c\times[\tau_0,0]\bigr).
\]
Proposition~\ref{com for weight} and the coverage argument in Step~4 yield
\[
\widetilde u(x,t)\geq c_h\bigl[A-K(N+F)\bigr]x_n^\sigma
\quad\text{on }\overline{Q_{1/2}^+}.
\]
Consequently, on $\overline{Q_{1/2}^+}$,
\begin{align*}
u(x,t)
&=\widetilde u(x,t)-\frac{F}{c_*}P(x)\\
&\geq\left[c_hA-Kc_h(N+F)-\frac{F}{c_*}\right]x_n^\sigma\\
&\geq\bigl[c_hA-c_H(N+F)\bigr]x_n^\sigma,
\end{align*}
after decreasing $c_h$ if necessary and choosing $c_H\geq Kc_h+c_*^{-1}$. This proves \eqref{hpfineq}.
\end{proof}

\subsection{Oscillation decay and boundary expansions}
We now combine Lemmas~\ref{nonhmlip} and~\ref{tailhpf} to show
that the oscillation of $u/x_n^\sigma$ decays at a geometric rate
in shrinking cylinders. This gives an expansion of $u$ near
$(0,0)$ as a constant multiple of $x_n^\sigma$, with a remainder
of higher order. The resulting estimate will be used to prove
the Liouville theorem below.

\begin{Proposition}[Boundary expansion with a finite tail]\label{flat}
Let $u$ be continuous on $\overline{Q_1^+}$ and satisfy \eqref{eq:tail-class}. Suppose, for some $F\geq0$, that
\[
\partial_tu-x_n^{\sigma\gamma}M^-u\geq-F,\qquad
\partial_tu-x_n^{\sigma\gamma}M^+u\leq F
\quad\text{in }Q_1^+,
\]
and $u=0$ on $\{x_n\leq0\}\times[-1,0]$. Assume
\[
N:=\|u\|_{L^\infty(B_2\times[-1,0])}
+\sup_{t\in[-1,0]}\int_{|y|>2}
\frac{|u(y,t)|}{|y|^{n+2\sigma}}\,dy<\infty.
\]
There are structural constants $\widetilde\alpha\in(0,\sigma(1-\gamma))$ and $\widetilde C>0$, and a number $\widetilde a\in\R$, such that
\[
|u(x,t)-\widetilde a x_n^\sigma|
\leq\widetilde C(N+F)
\left(|x|+|t|^{1/[\sigma(2-\gamma)]}\right)^{\widetilde\alpha}x_n^\sigma
\quad\text{on }\overline{Q_{1/2}^+},
\]
and $|\widetilde a|\leq\widetilde C(N+F)$. 
\end{Proposition}

\begin{proof}
If $N+F=0$, take $\widetilde a=0$. Otherwise, divide $u$ and $F$ by $N+F$ and retain the same notation, so that $N+F=1$. Set $r_k=2^{-k}$, $k\geq1$. We construct
\[
a_{k+1}\geq a_k,\qquad b_{k+1}\leq b_k,
\]
such that
\begin{equation}\label{eq:brackets}
\begin{split}
a_kx_n^\sigma&\leq u(x,t)\leq b_kx_n^\sigma
\quad\text{on }\overline{Q_{r_k}^+},\\
0&\leq b_k-a_k\leq\mu r_k^{\widetilde\alpha}.
\end{split}
\end{equation}
Here $0<\widetilde\alpha<\min\{\sigma/2,\sigma(1-\gamma)\}$ and $\mu>0$ will be fixed in Step~4.

\medskip

\textbf{Step 1.}
Since $N+F=1$ and $r_1=1/2$,
\begin{align*}
\|u\|_{L^\infty(B_1\times[-1,0])}
&\leq\|u\|_{L^\infty(B_2\times[-1,0])}\leq1,\\
\sup_{\tau\in[-1,0]}\int_{|z|\geq r_1}
\frac{|u(z,\tau)|}{|z|^{n+2\sigma}}\,dz
&\leq C\|u\|_{L^\infty(B_2\times[-1,0])}+\sup_{\tau\in[-1,0]}\int_{|z|>2}
\frac{|u(z,\tau)|}{|z|^{n+2\sigma}}\,dz\leq CN\leq C.
\end{align*}
Lemma~\ref{nonhmlip} gives
\[
|u(x,t)|\leq\overline Cx_n^\sigma
\quad\text{on }\overline{Q_{1/2}^+}.
\]
Choose
\[
a_1=-\overline C,\qquad b_1=\overline C,\qquad
\mu\geq2^{1+\widetilde\alpha}\overline C.
\]
Then
\[
b_1-a_1=2\overline C\leq\mu r_1^{\widetilde\alpha}.
\]

\textbf{Step 2. }
Assume \eqref{eq:brackets} for $1\leq j\leq k$, with
\[
-\overline C=a_1\leq a_2\leq\cdots\leq a_k
\leq b_k\leq\cdots\leq b_2\leq b_1=\overline C.
\]
If $b_k-a_k\leq\mu r_{k+1}^{\widetilde\alpha}$, set
\[
a_{k+1}=a_k,\qquad b_{k+1}=b_k.
\]
Otherwise,
\[
b_k-a_k>\mu r_{k+1}^{\widetilde\alpha}.
\]
First suppose
\[
u\left(\frac{r_k e_n}{2},
-r_k^{\sigma(2-\gamma)}\left(\frac23\right)^{\sigma(2-\gamma)}\right)
\geq\left(\frac{r_k}{2}\right)^\sigma\frac{a_k+b_k}{2}.
\]
Define
\[
w(x,t)=\frac{u(r_kx,r_k^{\sigma(2-\gamma)}t)
-a_k(r_kx_n)_+^\sigma}{r_k^{\sigma+\widetilde\alpha}}.
\]
By homogeneity and \eqref{eq:subtract-power},
\[
\begin{cases}
\partial_tw-x_n^{\sigma\gamma}M^-w\geq-F_k,\\
\partial_tw-x_n^{\sigma\gamma}M^+w\leq F_k,
\end{cases}
\qquad\text{in }Q_1^+,
\]
where $w\geq0$ on $\overline{Q_1^+}$, $w=0$ on $\{x_n\leq0\}\times[-1,0]$ and $F_k:=Fr_k^{\sigma(1-\gamma)-\widetilde\alpha}\leq F\leq1$. 
Moreover,
\begin{equation}\label{eq:midpoint-gain}
w\left(\frac{e_n}{2},-\left(\frac23\right)^{\sigma(2-\gamma)}\right)
\geq2^{-\sigma-1}\frac{b_k-a_k}{r_k^{\widetilde\alpha}}
>2^{-\sigma-1-\widetilde\alpha}\mu.
\end{equation}

\textbf{Step 3. }
Fix $t\in[-1,0]$. For $1<|y|<r_1/r_k$ and $y_n>0$, choose $1\leq j<k$ such that $r_{j+1}\leq r_k|y|\leq r_j.$ 
Then
\[
|r_ky|\leq r_j,\qquad
-r_j^{\sigma(2-\gamma)}
\leq-r_k^{\sigma(2-\gamma)}
\leq r_k^{\sigma(2-\gamma)}t\leq0,
\]
so that $(r_ky,r_k^{\sigma(2-\gamma)}t)\in\overline{Q_{r_j}^+}$. By \eqref{eq:brackets},
\[
w(y,t)\geq-(a_k-a_j)r_k^{-\widetilde\alpha}(y_n)_+^\sigma,
\]
and
\begin{align*}
0\leq a_k-a_j
&\leq(b_j-a_j)-(b_k-a_k)\leq\mu\bigl(r_j^{\widetilde\alpha}-r_{k+1}^{\widetilde\alpha}\bigr).
\end{align*}
Since $r_j\leq2r_k|y|$,
\begin{align*}
w^-(y,t)
&\leq\mu\bigl(2^{\widetilde\alpha}|y|^{\widetilde\alpha}
-2^{-\widetilde\alpha}\bigr)(y_n)_+^\sigma\leq\mu\bigl(2^{\widetilde\alpha}|y|^{\sigma+\widetilde\alpha}
-2^{-\widetilde\alpha}|y|^\sigma\bigr).
\end{align*}
For $y_n\leq0$, $w^-(y,t)=0$. Therefore,
\begin{align*}
\int_{1<|y|<r_1/r_k}
\frac{w^-(y,t)}{|y|^{n+2\sigma}}\,dy
&\leq C_n\mu\int_1^\infty
\bigl(2^{\widetilde\alpha}r^{-1-\sigma+\widetilde\alpha}
-2^{-\widetilde\alpha}r^{-1-\sigma}\bigr)\,dr\notag\\
&=C_n\mu\left(\frac{2^{\widetilde\alpha}}{\sigma-\widetilde\alpha}
-\frac{2^{-\widetilde\alpha}}\sigma\right).
\end{align*}
For $|y|\geq r_1/r_k$, using $z=r_ky$ and the tail bound in Step~1 gives
\begin{align*}
\int_{|y|\geq r_1/r_k}\frac{w^-(y,t)}{|y|^{n+2\sigma}}\,dy
&\leq r_k^{\sigma-\widetilde\alpha}
\left[\int_{|z|\geq r_1}
\frac{|u(z,r_k^{\sigma(2-\gamma)}t)|}{|z|^{n+2\sigma}}\,dz
+C|a_k|\right]\notag\\
&\leq Cr_k^{\sigma-\widetilde\alpha}(1+\overline C)\notag\\
&\leq C.
\end{align*}
Consequently,
\[
\sup_{t\in[-1,0]}\int_{|y|>1}
\frac{w^-(y,t)}{|y|^{n+2\sigma}}\,dy
\leq C_n\mu\left(\frac{2^{\widetilde\alpha}}{\sigma-\widetilde\alpha}
-\frac{2^{-\widetilde\alpha}}\sigma\right)+C,
\]
where $C=C(n,\sigma,\gamma,\lambda,\Lambda)$, $k\geq1$, $0<\widetilde\alpha\leq\sigma/2$.

\medskip

\textbf{Step 4. }
As $\widetilde\alpha\downarrow0$, $\frac{2^{\widetilde\alpha}}{\sigma-\widetilde\alpha}
-\frac{2^{-\widetilde\alpha}}\sigma \rightarrow0$, $1-2^{-\widetilde\alpha}\rightarrow0$ and $c_h2^{-\sigma-1-\widetilde\alpha}
\rightarrow c_h2^{-\sigma-1}>0$. Choose $0<\widetilde\alpha<\min\{\sigma/2,\sigma(1-\gamma)\}$ such that
\[
c_HC_n\left(\frac{2^{\widetilde\alpha}}{\sigma-\widetilde\alpha}
-\frac{2^{-\widetilde\alpha}}\sigma\right)
+(1-2^{-\widetilde\alpha})
\leq\frac12c_h2^{-\sigma-1-\widetilde\alpha},
\]
and then choose $\mu$ such that
\[
\mu\geq2^{1+\widetilde\alpha}\overline C,\qquad
\frac{c_H(C+1)}{\mu}\leq\frac12c_h2^{-\sigma-1-\widetilde\alpha}.
\]
Thus
\begin{equation}\label{iteration}
\begin{split}
&c_H\left[C_n\mu\left(\frac{2^{\widetilde\alpha}}{\sigma-\widetilde\alpha}
-\frac{2^{-\widetilde\alpha}}\sigma\right)+C+1\right]\leq\mu\left[c_h2^{-\sigma-1-\widetilde\alpha}
-(1-2^{-\widetilde\alpha})\right].
\end{split}
\end{equation}
Lemma~\ref{tailhpf}, $F_k\leq1$, \eqref{eq:midpoint-gain}, and \eqref{iteration} yield, on $\overline{Q_{1/2}^+}$,
\begin{align*}
w(x,t)
&\geq c_h w\left(\frac{e_n}{2},
-\left(\frac23\right)^{\sigma(2-\gamma)}\right)x_n^\sigma-c_Hx_n^\sigma\left(F_k+\sup_{\tau\in[-1,0]}
\int_{|y|>1}\frac{w^-(y,\tau)}{|y|^{n+2\sigma}}\,dy\right)\\
&\geq\mu(1-2^{-\widetilde\alpha})x_n^\sigma.
\end{align*}
Hence, on $\overline{Q_{r_{k+1}}^+}$,
\begin{align*}
u(x,t)-a_kx_n^\sigma
&=r_k^{\sigma+\widetilde\alpha}
w\left(\frac{x}{r_k},\frac{t}{r_k^{\sigma(2-\gamma)}}\right)\geq\mu(1-2^{-\widetilde\alpha})
r_k^{\widetilde\alpha}x_n^\sigma.
\end{align*}
Set
\[
a_{k+1}=a_k+\mu(1-2^{-\widetilde\alpha})r_k^{\widetilde\alpha},
\qquad b_{k+1}=b_k.
\]
Then
\begin{align*}
a_{k+1}x_n^\sigma&\leq u(x,t)\leq b_{k+1}x_n^\sigma
\quad\text{on }\overline{Q_{r_{k+1}}^+},\\
a_k&\leq a_{k+1}\leq b_{k+1}=b_k,
\end{align*}
and
\begin{align*}
0\leq b_{k+1}-a_{k+1}
&=b_k-a_k-\mu(1-2^{-\widetilde\alpha})r_k^{\widetilde\alpha}\leq\mu2^{-\widetilde\alpha}r_k^{\widetilde\alpha}
=\mu r_{k+1}^{\widetilde\alpha}.
\end{align*}

If instead
\[
u\left(\frac{r_k e_n}{2},
-r_k^{\sigma(2-\gamma)}\left(\frac23\right)^{\sigma(2-\gamma)}\right)
<\left(\frac{r_k}{2}\right)^\sigma\frac{a_k+b_k}{2},
\]
apply the same argument to $-u$, using
\[
-b_jx_n^\sigma\leq-u(x,t)\leq-a_jx_n^\sigma
\quad\text{on }\overline{Q_{r_j}^+},\qquad1\leq j\leq k.
\]
Equivalently, set
\[
w(x,t)=\frac{b_k(r_kx_n)_+^\sigma
-u(r_kx,r_k^{\sigma(2-\gamma)}t)}{r_k^{\sigma+\widetilde\alpha}}.
\]
Then
\[
u(x,t)\leq
\bigl[b_k-\mu(1-2^{-\widetilde\alpha})r_k^{\widetilde\alpha}\bigr]x_n^\sigma
\quad\text{on }\overline{Q_{r_{k+1}}^+}.
\]
Taking
\[
a_{k+1}=a_k,\qquad
b_{k+1}=b_k-\mu(1-2^{-\widetilde\alpha})r_k^{\widetilde\alpha},
\]
gives
\begin{align*}
a_{k+1}x_n^\sigma&\leq u(x,t)\leq b_{k+1}x_n^\sigma
\quad\text{on }\overline{Q_{r_{k+1}}^+},\\
a_k&=a_{k+1}\leq b_{k+1}\leq b_k,\\
0&\leq b_{k+1}-a_{k+1}\leq\mu r_{k+1}^{\widetilde\alpha}.
\end{align*}
This proves \eqref{eq:brackets} for every $k\geq1$.

\textbf{Step 5.}
Since
\begin{align*}
-\overline C&\leq a_k\leq a_{k+1}\leq b_{k+1}\leq b_k\leq\overline C,\\
b_k-a_k&\leq\mu r_k^{\widetilde\alpha}\longrightarrow0,
\end{align*}
we have
\[
\widetilde a:=\lim_{k\to\infty}a_k=\lim_{k\to\infty}b_k,\qquad
a_k\leq\widetilde a\leq b_k,\qquad |\widetilde a|\leq\overline C.
\]
For $(x,t)\in\overline{Q_{1/2}^+}\setminus\{(0,0)\}$, choose $k\geq1$ such that
\[
r_{k+1}<\max\{|x|,|t|^{1/[\sigma(2-\gamma)]}\}\leq r_k.
\]
Then
\begin{align*}
|u(x,t)-\widetilde a x_n^\sigma|
&\leq(b_k-a_k)x_n^\sigma
\leq\mu r_k^{\widetilde\alpha}x_n^\sigma\leq2^{\widetilde\alpha}\mu
\left(|x|+|t|^{1/[\sigma(2-\gamma)]}\right)^{\widetilde\alpha}x_n^\sigma.
\end{align*}
At $(0,0)$, both sides vanish. Restoring the normalization gives
\begin{align*}
|\widetilde a|&\leq\overline C(N+F),\\
|u(x,t)-\widetilde a x_n^\sigma|&\leq\widetilde C(N+F)
\left(|x|+|t|^{1/[\sigma(2-\gamma)]}\right)^{\widetilde\alpha}x_n^\sigma,\quad \mbox{for }(x,t)\in\overline{Q_{1/2}^+}.
\end{align*}
This finishes the proof.
\end{proof}

\subsection{A Liouville theorem}
We next apply Proposition~\ref{flat} to ancient solutions in the
half-space. Under the growth condition below, rescaling on larger
cylinders makes the remainder in the boundary expansion tend to
zero. The solution is therefore a constant multiple of
$(x_n)_+^\sigma$. This classification will be used to identify
the limits of the rescaled solutions in the proof of
Theorem~\ref{point}.

\begin{Theorem}[Liouville theorem]\label{liou}
Let $v$ be continuous on $\R^n\times(-\infty,0]$, zero on $\{x_n\leq0\}\times(-\infty,0]$, and satisfy
\[
\partial_tv-x_n^{\sigma\gamma}M^-v\geq0,\qquad
\partial_tv-x_n^{\sigma\gamma}M^+v\leq0
\quad\text{in }\R^n_+\times(-\infty,0].
\]
Let $\widetilde\alpha$ be the exponent in Proposition~\ref{flat}. Assume, for some $0<\alpha<\widetilde\alpha$, that
\[
\|v\|_{L^\infty(Q_R^+)}\leq CR^{\sigma+\alpha}\quad(R\geq1).
\]
Then $v(x,t)=b(x_n)_+^\sigma$ for a constant $b\in\R$.
\end{Theorem}
\begin{proof}
For $R\geq1$, set
\[
v_R(x,t)=R^{-\sigma-\alpha}v(Rx,R^{\sigma(2-\gamma)}t).
\]
By homogeneity,
\[
\begin{cases}
\partial_tv_R-x_n^{\sigma\gamma}M^-v_R\geq0,\\
\partial_tv_R-x_n^{\sigma\gamma}M^+v_R\leq0,
\end{cases}
\qquad\text{in }\R^n_+\times(-\infty,0],
\]
and 
$v_R=0$ on $\{x_n\leq0\}\times(-\infty,0].$
Moreover, for $R,\rho\geq1$,
\begin{align*}
\|v_R\|_{L^\infty(Q_\rho^+)}
&=R^{-\sigma-\alpha}\|v\|_{L^\infty(Q_{R\rho}^+)}
\leq C\rho^{\sigma+\alpha},\\
\|v_R\|_{L^\infty(B_2\times[-1,0])}
&\leq\|v_R\|_{L^\infty(Q_2^+)}\leq C.
\end{align*}
Since $B_{2^{j+1}}\times[-1,0]\subset Q_{2^{j+1}}$ for $j\geq1$,
\begin{align*}
\sup_{t\in[-1,0]}\int_{|y|>2}
\frac{|v_R(y,t)|}{|y|^{n+2\sigma}}\,dy
&\leq\sum_{j=1}^\infty\sup_{t\in[-1,0]}
\int_{2^j<|y|\leq2^{j+1}}
\frac{|v_R(y,t)|}{|y|^{n+2\sigma}}\,dy\\
&\leq C\sum_{j=1}^\infty2^{-2j\sigma}
\|v_R\|_{L^\infty(Q_{2^{j+1}}^+)}\\
&\leq C\sum_{j=1}^\infty2^{-j(\sigma-\alpha)}\\
&\leq C_0,
\end{align*}
where $0<\alpha<\widetilde\alpha<\sigma$. Thus
\[
\|v_R\|_{L^\infty(B_2\times[-1,0])}
+\sup_{t\in[-1,0]}\int_{|y|>2}
\frac{|v_R(y,t)|}{|y|^{n+2\sigma}}\,dy
\leq C_0\qquad(R\geq1).
\]
Also,
\begin{align*}
|v_R(y,t)|&\leq C(1+|y|)^{\sigma+\alpha}
\quad\bigl((y,t)\in\R^n\times[-1,0]\bigr),\\
\int_{\R^n}\frac{(1+|y|)^{\sigma+\alpha}}{1+|y|^{n+2\sigma}}\,dy&<\infty.
\end{align*}
By continuity and dominated convergence,
\[
\int_{\R^n}\frac{|v_R(y,t)-v_R(y,s)|}{1+|y|^{n+2\sigma}}\,dy
\longrightarrow0\qquad(s\to t\text{ in }[-1,0]).
\]

Proposition~\ref{flat}, with $F=0$, gives $a_R\in\R$ such that
\begin{align*}
|a_R|&\leq C_0,\\
|v_R(x,t)-a_Rx_n^\sigma|&\leq C_0\left(|x|+|t|^{1/[\sigma(2-\gamma)]}\right)^{\widetilde\alpha}x_n^\sigma
\quad\text{on }\overline{Q_{1/2}^+}.
\end{align*}
Consequently, for $(x,t)\in\overline{Q_{R/2}^+}$,
\begin{align*}
|v(x,t)-a_RR^\alpha x_n^\sigma|
&=R^{\sigma+\alpha}
\left|v_R\left(\frac{x}{R},\frac{t}{R^{\sigma(2-\gamma)}}\right)
-a_R\left(\frac{x_n}{R}\right)^\sigma\right|\\
&\leq C_0R^{\alpha-\widetilde\alpha}
\left(|x|+|t|^{1/[\sigma(2-\gamma)]}\right)^{\widetilde\alpha}x_n^\sigma.
\end{align*}
Set $b:=2^\sigma v(e_n/2,0)$. For $R\geq1$,
\begin{align*}
|a_RR^\alpha-b|
&=2^\sigma\left|v(e_n/2,0)-a_RR^\alpha2^{-\sigma}\right|\\
&\leq C_0 2^{-\widetilde\alpha}R^{\alpha-\widetilde\alpha}\\
&\longrightarrow0\qquad\mbox{as }R\to\infty.
\end{align*}
For any fixed $(x,t)\in\R^n_+\times(-\infty,0]$, if $R>2\max\left\{1,|x|,|t|^{1/[\sigma(2-\gamma)]}\right\}$, 
then $(x,t)\in Q_{R/2}^+$,
and hence
\begin{align*}
|v(x,t)-bx_n^\sigma|
&\leq|v(x,t)-a_RR^\alpha x_n^\sigma|
+|a_RR^\alpha-b|x_n^\sigma\\
&\leq\left[C_0R^{\alpha-\widetilde\alpha}
\left(|x|+|t|^{1/[\sigma(2-\gamma)]}\right)^{\widetilde\alpha}
+|a_RR^\alpha-b|\right]x_n^\sigma\\
&\longrightarrow0\qquad\mbox{as }R\to\infty.
\end{align*}
Therefore,
\[
v(x,t)=
\begin{cases}
bx_n^\sigma,&x_n>0,\\
0,&x_n\leq0,
\end{cases}
\quad\mbox{for all }(x,t)\in\R^n\times(-\infty,0].
\]
This finishes the proof.
\end{proof}
\section{Boundary estimates in $C^{1,1}$ domains}\label{sec:Boundary estimate}
The Liouville theorem classifies the possible blow-up limits. We now establish the boundary estimates needed for compactness in curved domains. 

\subsection{Boundary growth estimates}

\begin{Lemma}\label{bdryreg}
Let $0\in\partial\Omega$, and let $\Omega$ be uniformly $C^{1,1}$ near $\partial\Omega\cap B_3$. Let $v:\R^n\times[-1,0]\to\R$ satisfy \eqref{eq:tail-class}, be continuous on
$(\overline\Omega\cap\overline B_2)\times[-1,0]$, and vanish on $\Omega^c\times[-1,0]$. Suppose
\[
\partial_tv-d^{\sigma\gamma}M^+v\leq F,\qquad
\partial_tv-d^{\sigma\gamma}M^-v\geq-F
\quad\text{in }(\Omega\cap B_2)\times(-1,0],
\]
where $F\geq0$. Set
\[
N=\|v\|_{L^\infty(B_2\times[-1,0])}
+\sup_{t\in[-1,0]}\int_{|y|>2}\frac{|v(y,t)|}{|y|^{n+2\sigma}}\,dy.
\]
Then
\[
|v(x,t)|\leq C(N+F)d(x)^\sigma
\quad\text{on }(\overline\Omega\times[-1,0])\cap\overline{Q_{1/2}}.
\]
The positive constant $C$ depends only on the structural parameters and the stated local geometric bounds. 
\end{Lemma}

\begin{proof}
If $N+F=0$, then $v=0$ on $(\overline\Omega\times[-1,0])\cap\overline{Q_{1/2}}$. Assume $N+F>0$.

\medskip

\textbf{Step 1. }
Choose $\chi\in C_c^\infty(B_2)$ with $0\leq\chi\leq1$ and $\chi\equiv 1$ on $B_{3/2}$, and set
\[
\bar v=\chi v.
\] 
Since $0\in\partial\Omega$, for $x\in\Omega\cap B_1$ we have $0<d(x)=\inf_{z\in\partial\Omega}|x-z|\leq|x|<1$ and $|x-y|\geq|y|-1\geq\frac13|y|$ whenever $|y|>3/2$. 
Thus, uniformly in $L\in\mathcal L_*$, we have
\begin{align*}
|L(\bar v-v)(x,t)|
&\leq C\int_{|y|>3/2}\frac{|v(y,t)|}{|x-y|^{n+2\sigma}}\,dy\\
&\leq C\|v\|_{L^\infty(B_2\times[-1,0])}
+C\int_{|y|>2}\frac{|v(y,t)|}{|y|^{n+2\sigma}}\,dy\\
&\leq CN.
\end{align*}
Consequently,
\[
\begin{cases}
\partial_t\bar v-d^{\sigma\gamma}M^+\bar v\leq F+CN,\\
\partial_t\bar v-d^{\sigma\gamma}M^-\bar v\geq-F-CN,
\end{cases}
\quad\text{in }(\Omega\cap B_1)\times(-1,0],
\]
and $\|\bar v\|_{L^\infty(\R^n\times[-1,0])}\leq N.$ 
We normalize by $C(N+F)$ and retain the notation $\bar v$:
\[
\bar v=\frac{\chi v}{C(N+F)},\qquad
\|\bar v\|_{L^\infty(\R^n\times[-1,0])}\leq1,
\]
and
\[
\begin{cases}
\partial_t\bar v-d^{\sigma\gamma}M^+\bar v\leq1,\\
\partial_t\bar v-d^{\sigma\gamma}M^-\bar v\geq-1,
\end{cases}
\quad\text{in }(\Omega\cap B_1)\times(-1,0].
\]
Set
\[
\widetilde v(x,t)=(t+1)\bar v(x,t),\qquad
\widetilde v(\cdot,-1)=0.
\]
Let $\widetilde\varphi$ touch $\widetilde v$ from above at $(x_0,t_0)$. Using the same patch cylinder, set
\[
\varphi=\frac{\widetilde\varphi}{t+1}.
\]
Then $\varphi\geq\bar v$, $\varphi(x_0,t_0)=\bar v(x_0,t_0)$, $\widetilde v_{\widetilde\varphi}(\cdot,t_0)
=(t_0+1)\bar v_\varphi(\cdot,t_0)$,
and
\begin{align*}
\partial_t\widetilde\varphi(x_0,t_0)
-d(x_0)^{\sigma\gamma}M^+\widetilde v_{\widetilde\varphi}(x_0,t_0)
&=(t_0+1)\bigl[\partial_t\varphi
-d^{\sigma\gamma}M^+\bar v_\varphi\bigr](x_0,t_0)
+\bar v(x_0,t_0)\\
&\leq(t_0+1)+\bar v(x_0,t_0)\\
&\leq2.
\end{align*}
The same calculation for $-\bar v$, using $M^+(-\bar v)=-M^-\bar v$, gives
\[
\partial_t(\pm\widetilde v)
-d^{\sigma\gamma}M^+(\pm\widetilde v)\leq2
\quad\text{in }(\Omega\cap B_1)\times(-1,0].
\]
For $s,t\in[-1,0]$, \eqref{eq:tail-class} gives
\begin{align*}
\int_{\R^n}\frac{|\bar v(y,t)-\bar v(y,s)|}{1+|y|^{n+2\sigma}}\,dy
&\leq\frac{1}{C(N+F)}
\int_{\R^n}\frac{|v(y,t)-v(y,s)|}{1+|y|^{n+2\sigma}}\,dy
\longrightarrow0,\\
\int_{\R^n}\frac{|\widetilde v(y,t)-\widetilde v(y,s)|}{1+|y|^{n+2\sigma}}\,dy
&\leq\int_{\R^n}\frac{|\bar v(y,t)-\bar v(y,s)|}{1+|y|^{n+2\sigma}}\,dy
+C|t-s|\longrightarrow0
\end{align*}
as $s\to t$. Also,
\[
\|\widetilde v(\cdot,t)\|_{L^\infty(\R^n)}
\leq t+1\longrightarrow0\qquad\mbox{as }t\downarrow-1.
\]

\textbf{Step 2. }
Set
\[
D(x)=\operatorname{dist}(x,\Omega^c),\qquad
\beta=\sigma(2-\gamma)\in(\sigma,2\sigma),\qquad
s(x)=D(x)-\operatorname{dist}(x,\Omega).
\]
Then
\[
D=s_+=
\begin{cases}
d&\text{in }\Omega,\\
0&\text{in }\Omega^c,
\end{cases}
\qquad\mbox{and }
|D(x)-D(y)|\leq|x-y|\quad(x,y\in\R^n).
\]
Choose $\ell\in(0,1/4)$, depending only on the local geometric bounds, such that
\[
s\in C^{1,1}\bigl(B_2\cap\{|s|<2\ell\}\bigr),\qquad
[\nabla s]_{C^{0,1}(B_2\cap\{|s|<2\ell\})}\leq C.
\]
Define
\[
f(x)=\min\{D(x),\ell\}^{\sigma},\qquad
g(x)=\min\{D(x),\ell\}^{\beta},\qquad
\Phi=2f-g.
\]
We have
$
r^\sigma\leq2r^\sigma-r^\beta\leq2r^\sigma$
for $0\leq r\leq\ell<1$, $\Phi\in C(\R^n)$, $\Phi=0\text{ in }\Omega^c$,
\begin{gather*}
\min\{D(x),\ell\}^\sigma\leq\Phi(x)\leq2D(x)^\sigma,\qquad
0\leq\Phi\leq2\ell^\sigma.
\end{gather*}
For $x_0\in\Omega\cap B_1$ with $d(x_0)<\ell/4$ and $|z|<\ell/2$,
\[
|x_0+z|<1+\ell/2<2,\qquad
|s(x_0+z)|\leq d(x_0)+|z|<3\ell/4,
\]
and hence
\[
B_{\ell/2}(x_0)\subset B_2\cap\{|s|<2\ell\}.
\]

\textbf{Step 3. }
Fix $x_0\in\Omega\cap B_1$, $\rho=d(x_0)\in(0,\ell/4)$, $z_0\in\partial\Omega$ and $|x_0-z_0|=\rho$.
Then $|z_0|\leq|x_0|+\rho<1+\ell/4<2.$ 
After a translation and rotation, we retain the same notation and have
\[
z_0=0,\qquad x_0=\rho e_n,\qquad
s(x_0)=\rho,\qquad
\nabla s(x_0)=\frac{x_0-z_0}{\rho}=e_n.
\]
The class $\mathcal L_*$ is invariant under these transformations. Set
\[
h_1(x)=(x_n)_+^\sigma,\qquad h_2(x)=(x_n)_+^\beta.
\]
For $|z|<\ell/2$,
\begin{align*}
|s(x_0\pm z)-(\rho\pm z_n)|&\leq C|z|^2,\\
|D(x_0\pm z)-(\rho\pm z_n)_+|
&=|s(x_0\pm z)_+-(\rho\pm z_n)_+|
\leq C|z|^2,\\
0\leq D(x_0\pm z)&\leq\rho+|z|<3\ell/4<\ell,\\
f(x_0\pm z)&=D(x_0\pm z)^\sigma,\\
g(x_0\pm z)&=D(x_0\pm z)^\beta,\\
(f-h_1)(x_0)&=(g-h_2)(x_0)=0.
\end{align*}
For $L\in\mathcal L_*$ and $q\in\{f-h_1,g-h_2\}$,
\[
|Lq(x_0)|
\leq\Lambda\left(
\int_{|z|\leq\rho/2}
+\int_{\rho/2<|z|<\ell/2}
+\int_{|z|\geq\ell/2}\right)
\frac{|\delta q(x_0,z)|}{|z|^{n+2\sigma}}\,dz.
\]

\emph{(i) $|z|\leq\rho/2$.}
In this region, $D(x_0\pm z),\ \rho\pm z_n\in[\rho/2,3\rho/2]$,
\begin{align*}
|(f-h_1)(x_0\pm z)|
&\leq C\rho^{\sigma-1}|D(x_0\pm z)-(\rho\pm z_n)|
\leq C\rho^{\sigma-1}|z|^2,\\
|(g-h_2)(x_0\pm z)|
&\leq C\rho^{\beta-1}|D(x_0\pm z)-(\rho\pm z_n)|
\leq C\rho^{\beta-1}|z|^2.
\end{align*}
Therefore,
\begin{align*}
\int_{|z|\leq\rho/2}
\frac{|\delta(f-h_1)(x_0,z)|}{|z|^{n+2\sigma}}\,dz
&\leq C\rho^{\sigma-1}\int_0^{\rho/2}r^{1-2\sigma}\,dr
\leq C\rho^{1-\sigma},\\
\int_{|z|\leq\rho/2}
\frac{|\delta(g-h_2)(x_0,z)|}{|z|^{n+2\sigma}}\,dz
&\leq C\rho^{\beta-1}\int_0^{\rho/2}r^{1-2\sigma}\,dr
\leq C\rho^{1-\sigma\gamma}.
\end{align*}

\emph{(ii) $\rho/2<|z|<\ell/2$.}
Here $0\leq D(x_0\pm z)$, $(\rho\pm z_n)_+\leq\rho+|z|<3|z|$,
\[
|(f-h_1)(x_0\pm z)|
\leq|D(x_0\pm z)-(\rho\pm z_n)_+|^\sigma
\leq C|z|^{2\sigma},
\]
and
\begin{align*}
|(g-h_2)(x_0\pm z)|
&\leq
\begin{cases}
|D(x_0\pm z)-(\rho\pm z_n)_+|^\beta,&\beta\leq1,\\
C|z|^{\beta-1}|D(x_0\pm z)-(\rho\pm z_n)_+|,&\beta>1,
\end{cases}\\
&\leq
\begin{cases}
C|z|^{2\beta},&\beta\leq1,\\
C|z|^{\beta+1},&\beta>1,
\end{cases}
\leq C|z|^{\beta+\sigma},
\end{align*}
since $\beta>\sigma$, $\sigma<1$, and $|z|<1$. Hence
\begin{align*}
\int_{\rho/2<|z|<\ell/2}
\frac{|\delta(f-h_1)(x_0,z)|}{|z|^{n+2\sigma}}\,dz
&\leq C\int_{\rho/2}^{\ell/2}\frac{dr}{r}
=C\log\frac{\ell}{\rho},\\
\int_{\rho/2<|z|<\ell/2}
\frac{|\delta(g-h_2)(x_0,z)|}{|z|^{n+2\sigma}}\,dz
&\leq C\int_{\rho/2}^{\ell/2}r^{\beta-\sigma-1}\,dr
\leq C\ell^{\beta-\sigma},
\end{align*}
with $\beta-\sigma=\sigma(1-\gamma)>0$.

\emph{(iii) $|z|\geq\ell/2$.}
Since $\rho<\ell/4\leq|z|/2$,
\begin{align*}
|(f-h_1)(x_0\pm z)|
&\leq\ell^\sigma+(\rho+|z|)^\sigma\leq C|z|^\sigma,\\
|(g-h_2)(x_0\pm z)|
&\leq\ell^\beta+(\rho+|z|)^\beta\leq C|z|^\beta.
\end{align*}
Thus
\begin{align*}
\int_{|z|\geq\ell/2}
\frac{|\delta(f-h_1)(x_0,z)|}{|z|^{n+2\sigma}}\,dz
&\leq C\int_{\ell/2}^{\infty}r^{-\sigma-1}\,dr
\leq C\ell^{-\sigma},\\
\int_{|z|\geq\ell/2}
\frac{|\delta(g-h_2)(x_0,z)|}{|z|^{n+2\sigma}}\,dz
&\leq C\int_{\ell/2}^{\infty}r^{\beta-2\sigma-1}\,dr
\leq C\ell^{-\sigma\gamma},
\end{align*}
with $\sigma>0$ and $2\sigma-\beta=\sigma\gamma>0$.

Combining (i)--(iii),
\begin{align*}
|L(f-h_1)(x_0)|
&\leq C\left(\rho^{1-\sigma}
+\log\frac{\ell}{\rho}+\ell^{-\sigma}\right)
\leq C_\ell(1+|\log\rho|),\\
|L(g-h_2)(x_0)|
&\leq C\left(\rho^{1-\sigma\gamma}
+\ell^{\beta-\sigma}+\ell^{-\sigma\gamma}\right)
\leq C_\ell,
\end{align*}
where $C_\ell=C_\ell(n,\sigma,\gamma,\lambda,\Lambda,\ell,
\text{local geometric bounds}).$ 
By Lemmas~2.2--2.3 in \cite{ros2016boundary},
\[
Lh_1(x_0)=0,\qquad
Lh_2(x_0)=c_L\rho^{-\sigma\gamma},\qquad
0<c_*\leq c_L\leq C_*
\quad(L\in\mathcal L_*),
\]
with structural $c_*,C_*>0$. Therefore,
\[
M^+f(x_0)\leq C_\ell(1+|\log\rho|),\qquad
M^-g(x_0)\geq c_*\rho^{-\sigma\gamma}-C_\ell,
\]
and, for $x\in\Omega\cap B_1$ with $0<d(x)<\ell/4$,
\begin{align*}
M^+\Phi(x)
&\leq2M^+f(x)-M^-g(x)\\
&\leq C_\ell(1+|\log d(x)|)-c_*d(x)^{-\sigma\gamma}.
\end{align*}

\textbf{Step 4. }
Since $r^{\sigma\gamma}(1+|\log r|)\rightarrow0$ as $r\downarrow0$, choose $\epsilon\in(0,\ell/4)$ with
\[
C_\ell r^{\sigma\gamma}(1+|\log r|)\leq c_*/2
\qquad(0<r<\epsilon).
\]
Set
\[
A_\epsilon=\Omega\cap B_1\cap\{d<\epsilon\}.
\]
Then
\[
-d^{\sigma\gamma}M^+\Phi\geq c_*/2
\quad\text{in }A_\epsilon.
\]
Choose $\rho\in C^\infty(\R^n)$ satisfying $0\leq\rho\leq1$,  $\rho=0$ on $B_{1/2}$, $\rho=1$ on $\R^n\setminus B_1$ and $\sup_{L\in\mathcal L_*}\|L\rho\|_{L^\infty(\R^n)}\leq C$, and take
\[
K\geq\max\left\{\epsilon^{-\sigma},
\frac{4+2C\epsilon^{\sigma\gamma}}{c_*}\right\}.
\]
For $x\in A_\epsilon$,
\begin{align*}
-d^{\sigma\gamma}M^+(K\Phi+\rho)
&\geq-Kd^{\sigma\gamma}M^+\Phi-d^{\sigma\gamma}M^+\rho\\
&\geq Kc_*/2-C\epsilon^{\sigma\gamma}\\
&\geq2.
\end{align*}
Moreover,
\[
K\Phi+\rho\in C^{1,1}_{\mathrm{loc}}(A_\epsilon)
\cap C(\R^n)\cap L^\infty(\R^n).
\]
At a backward lower contact $(x_0,t_0)$, since $\partial_t\varphi(x_0,t_0)\geq0$ and $M^+\bigl((K\Phi+\rho)_\varphi\bigr)(x_0,t_0)\leq M^+(K\Phi+\rho)(x_0)$,
we have
\[
\partial_t\varphi(x_0,t_0)
-d(x_0)^{\sigma\gamma}M^+\bigl((K\Phi+\rho)_\varphi\bigr)(x_0,t_0)
\geq-d(x_0)^{\sigma\gamma}M^+(K\Phi+\rho)(x_0)\geq2.
\]
The initial and exterior inequalities are $|\widetilde v(x,-1)|=0\leq K\Phi(x)+\rho(x)$ for $x\in\R^n$, and
\[
|\widetilde v(x,t)|\leq
\begin{cases}
1\leq\rho(x),&x\in\R^n\setminus B_1,\\
1\leq K\Phi(x),&x\in\Omega\cap B_1,\ d(x)\geq\epsilon,\\
0\leq K\Phi(x)+\rho(x),&x\in\Omega^c,
\end{cases}
\qquad t\in[-1,0].
\]
Corollary~\ref{com for dis}, applied to $\widetilde v$ and $-\widetilde v$, yields
\[
|\widetilde v|\leq K\Phi+\rho
\quad\text{on }\overline{A_\epsilon}\times[-1,0].
\]
On $(\Omega\times[-1,0])\cap\overline{Q_{1/2}}$, $\rho=0$, $t+1\geq1-2^{-\sigma(2-\gamma)}>0$, and
\[
|\bar v(x,t)|\leq
\begin{cases}
\displaystyle\frac{2K}{1-2^{-\sigma(2-\gamma)}}d(x)^\sigma,
&0<d(x)<\epsilon,\\[6pt]
1\leq\epsilon^{-\sigma}d(x)^\sigma,&d(x)\geq\epsilon.
\end{cases}
\]
Finally, $v=C(N+F)\bar v$ on $B_{1/2}\times[-1,0]$ and $v=0$ on $\partial\Omega\times[-1,0]$,
and hence
\[
|v(x,t)|\leq C(N+F)d(x)^\sigma
\quad\text{on }(\overline\Omega\times[-1,0])\cap\overline{Q_{1/2}}.
\]
\end{proof}

\subsection{Boundary compactness}

We establish the following compactness result, including convergence of the nonlocal terms.

\begin{Corollary}[Compactness under rescaling]\label{cor:compactness}
Let $T_j\to\infty$, let $\Omega_j\subset\R^n$ be domains, and set
\[
\begin{aligned}
d_j(x)&=\operatorname{dist}(x,\partial\Omega_j),\qquad x\in\Omega_j,\\
s_j(x)&=\operatorname{dist}(x,\Omega_j^c)-\operatorname{dist}(x,\Omega_j),
\qquad x\in\R^n.
\end{aligned}
\]
Assume that, for every $R>0$, there exist $r_R,C_R>0$ such that
\[
\|s_j-x_n\|_{L^\infty(B_R)}\longrightarrow0,\qquad
\sup_j\|s_j\|_{C^{1,1}(B_R\cap\{|s_j|<r_R\})}\leq C_R.
\]
Let $v_j:\R^n\times[-T_j,0]\to\R$ satisfy \eqref{eq:tail-class} with $I=[-T_j,0]$, and suppose $v_j|_{\overline\Omega_j\times[-T_j,0]}\in C(\overline\Omega_j\times[-T_j,0])$, $v_j=0$ on $\Omega_j^c\times[-T_j,0]$, and
\[
\begin{split}
\partial_tv_j-d_j^{\sigma\gamma}M^+v_j\leq F_j,\qquad
\partial_tv_j-d_j^{\sigma\gamma}M^-v_j\geq-F_j
\quad\text{in }\Omega_j\times(-T_j,0]
\end{split}
\]
in the viscosity sense. Assume that, for some $\alpha\in(0,\sigma)$, for all $S>0$ there exist $j_S\in\mathbb N,\ C_S>0$ such that $T_j>S$ when $j\geq j_S$,
\[
\begin{gathered}
|v_j(x,t)|\leq C_S(1+|x|)^{\sigma+\alpha}
\quad\mbox{for }j\geq j_S,\ (x,t)\in\R^n\times[-S,0],
\end{gathered}
\]
and
\[
F_j\geq0,\qquad
\sup_{(x,t)\in(\Omega_j\cap B_R)\times[-S,0]}F_j(x,t)
\longrightarrow0\quad(R,S>0).
\]
Then there exist a subsequence $v_{j_k}$ and $v\in C(\R^n\times(-\infty,0])$ such that, for every $R,S>0$,
\[
\begin{split}
\|v_{j_k}-v\|_{L^\infty(B_R\times[-S,0])}&\longrightarrow0,\\
\sup_{t\in[-S,0]}\int_{\R^n}
\frac{|v_{j_k}(x,t)-v(x,t)|}{1+|x|^{n+2\sigma}}\,dx
&\longrightarrow0.
\end{split}
\]
Moreover,
\[
v=0\quad\text{on }\{x_n\leq0\}\times(-\infty,0],
\]
and, in the viscosity sense,
\[
\partial_tv-x_n^{\sigma\gamma}M^+v\leq0,\qquad
\partial_tv-x_n^{\sigma\gamma}M^-v\geq0
\quad\text{in }\R^n_+\times(-\infty,0].
\]
\end{Corollary}

\begin{proof}
Extend $d_j=0$ in $\Omega_j^c$. Fix $R,S>0$.

\medskip

\textbf{Step 1. Uniform boundary and interior estimates.}
For all sufficiently large $j$, we have $T_j>S+1$,
\begin{align*}
\|F_j\|_{L^\infty((\Omega_j\cap B_{R+2})\times[-S-1,0])}&\leq1,\\
\sup_{t\in[-S-1,0]}
\left[
\|v_j(\cdot,t)\|_{L^\infty(B_{R+2})}
+\int_{\R^n}\frac{|v_j(y,t)|}{1+|y|^{n+2\sigma}}\,dy
\right]&\leq C_{R,S}.
\end{align*}
Lemma~\ref{bdryreg}, after translation and intrinsic rescaling, gives
\[
|v_j(x,t)|\leq C_{R,S}d_j(x)^\sigma
\quad\text{on }\overline B_{R+1}\times[-S,0].
\]
For each $\delta>0$, the H\"older estimate \eqref{eq:interior-holder} in Theorem~\ref{thm:interior-estimates}, gives $\beta\in(0,1)$ such that, whenever $x,y\in\overline B_R$, $s,t\in[-S,0]$, and $\min\{d_j(x),d_j(y)\}\geq\delta$, we have
\[
|v_j(x,t)-v_j(y,s)|
\leq C_{R,S,\delta}
\bigl(|x-y|+|t-s|^{1/(2\sigma)}\bigr)^\beta.
\]
On the other hand, if $|x-y|\leq\delta$ and $\min\{d_j(x),d_j(y)\}\leq\delta$, then
$\max\{d_j(x),d_j(y)\}\leq2\delta$,
and therefore
\[
|v_j(x,t)-v_j(y,s)|
\leq C_{R,S}\bigl(d_j(x)^\sigma+d_j(y)^\sigma\bigr)
\leq C_{R,S}\delta^\sigma.
\]
Combining these estimates, for $0<\eta\leq\delta$,
\[
\sup_{j\geq j_0}
\sup_{\substack{x,y\in\overline B_R,\ s,t\in[-S,0]\\
|x-y|+|t-s|^{1/(2\sigma)}\leq\eta}}
|v_j(x,t)-v_j(y,s)|
\leq C_{R,S}\delta^\sigma+C_{R,S,\delta}\eta^\beta.
\]
Letting first $\eta\downarrow0$ and then $\delta\downarrow0$ proves equicontinuity.

\textbf{Step 2. Local uniform and weighted convergence.}
By the Arzel\`a--Ascoli theorem and a diagonal argument, there exists $v\in C(\R^n\times(-\infty,0])$ such that, after passing to a subsequence,
$\|v_j-v\|_{L^\infty(\overline B_R\times[-S,0])}\rightarrow0$ for $R,S>0$. Moreover, $d_j\rightarrow(x_n)_+$ locally uniformly on $\R^n$, $|v(x,t)|\leq C_{R,S}(x_n)_+^\sigma$ on $\overline B_R\times[-S,0]$, and $v=0$ on $\{x_n\leq0\}\times(-\infty,0]$. The growth bound yields, for $R\geq1$,
\begin{align*}
\sup_{j\geq j_0}\sup_{t\in[-S,0]}
\int_{|y|>R}\frac{|v_j(y,t)|}{|y|^{n+2\sigma}}\,dy
&\leq C_S\int_R^\infty r^{\alpha-\sigma-1}\,dr\leq C_S R^{\alpha-\sigma}\longrightarrow0
\quad\mbox{as }R\to\infty.
\end{align*}
The same bound holds for $v$. Consequently,
\begin{align*}
\sup_{t\in[-S,0]}
\int_{\R^n}\frac{|v_j(y,t)-v(y,t)|}{1+|y|^{n+2\sigma}}\,dy
&\leq C\|v_j-v\|_{L^\infty(B_R\times[-S,0])}
+C_S R^{\alpha-\sigma}.
\end{align*}
Taking $j\to\infty$ and then $R\to\infty$ gives
\[
\sup_{t\in[-S,0]}
\int_{\R^n}\frac{|v_j(y,t)-v(y,t)|}{1+|y|^{n+2\sigma}}\,dy
\longrightarrow0.
\]

\textbf{Step 3. The limit equations.}
Let $\varphi$ touch $v$ strictly from above at $(x_0,t_0)\in\R^n_+\times(-\infty,0]$. Since $v_j\to v$ locally uniformly, there exist points $(x_j,t_j)$ and constants $c_j$ such that $(x_j,t_j)\rightarrow(x_0,t_0)$, $c_j\rightarrow0$, and $\varphi_j:=\varphi+c_j$ touches $v_j$ from above
at $(x_j,t_j)$.

Choose a fixed $r>0$ sufficiently small that $\overline{B_r(x_0)}\subset\R^n_+$ and
$\overline{B_r(x_j)}\subset\Omega_j$
for all $j\geq j_0$.
For the patches in Definition~\ref{defn:vis},
\begin{align*}
\int_{|z|<r}
\frac{|\delta\varphi_j(x_j,z,t_j)-\delta\varphi(x_0,z,t_0)|}
{|z|^{n+2\sigma}}\,dz&\longrightarrow0,\\
\sup_{r\leq|z|\leq R}
|\delta v_j(x_j,z,t_j)-\delta v(x_0,z,t_0)|&\longrightarrow0,\\
\sup_{j\geq j_0}\int_{|z|>R}
\frac{|\delta v_j(x_j,z,t_j)|+|\delta v(x_0,z,t_0)|}
{|z|^{n+2\sigma}}\,dz&\leq CR^{\alpha-\sigma}.
\end{align*}
Consequently,
\begin{align*}
&\limsup_{j\to\infty}\sup_{L\in\mathcal L_*}
\left|
L(v_j)_{\varphi_j,r}(x_j,t_j)
-Lv_{\varphi,r}(x_0,t_0)
\right|\leq C R^{\alpha-\sigma}
\longrightarrow0
\qquad\text{as }R\to\infty.
\end{align*}
As in Theorem~5.3 of \cite{changstability},
\begin{align*}
\partial_t\varphi(x_0,t_0)
&=\lim_{j\to\infty}\partial_t\varphi_j(x_j,t_j)\\
&\leq\lim_{j\to\infty}
\left[d_j(x_j)^{\sigma\gamma}
M^+(v_j)_{\varphi_j,r}(x_j,t_j)+F_j(x_j,t_j)\right]\\
&=x_{0,n}^{\sigma\gamma}M^+v_{\varphi,r}(x_0,t_0).
\end{align*}
If $\varphi$ touches $v$ from below at $(x_0,t_0)$, the same argument gives
\[
\partial_t\varphi(x_0,t_0) \geq x_{0,n}^{\sigma\gamma}M^-v_{\varphi,r}(x_0,t_0).
\]
Hence
\[
\begin{cases}
\partial_tv-x_n^{\sigma\gamma}M^+v\leq0,\\
\partial_tv-x_n^{\sigma\gamma}M^-v\geq0,
\end{cases}
\quad\text{in }\R^n_+\times(-\infty,0].
\]
\end{proof}
\section{Boundary expansions and proofs of the main results}\label{sec:ratio}
Let $C_1,C_2$ be the constants in Theorem~\ref{thm1}. Throughout this section, the auxiliary extremal operators are taken with the enlarged bounds
\[
\lambda_0=\lambda\min\{1,C_1^\gamma\},\qquad
\Lambda_0=\Lambda\max\{1,C_2^\gamma\}.
\]
We retain the notation $M^\pm$. Indeed, if $C_1\leq u/d^\sigma\leq C_2$, then
\[
u^\gamma L=d^{\sigma\gamma}\left(\frac{u}{d^\sigma}\right)^\gamma L,
\]
and the resulting stable operator belongs pointwise to the enlarged class. The exponent $\widetilde\alpha$ from Section~\ref{sec:liouville}, the boundary estimates, and the Liouville theorem are all used with these same bounds. Their constants and exponents may consequently depend on $T,c_1,c_2$ and $\Omega$.

Let $\psi$ be the first eigenfunction in Lemma \ref{firsteigenfunction}. We will subtract multiples of $\psi$ from viscosity solutions, since the regularity of $\psi$ suffices for the eigenfunction equation to hold classically. 

\subsection{Approximation by the first eigenfunction}
We now combine Corollary~\ref{cor:compactness} and
Theorem~\ref{liou} to obtain a boundary expansion for the solution $u$ in Theorem~\ref{mainthm}.
Near each boundary point and fixed time, $u$ can be approximated
by a constant multiple of the first eigenfunction $\psi$, with
a remainder of higher order. This expansion will be used to
prove the H\"older continuity of $u/\psi$.

\begin{Theorem}[Boundary expansion]\label{point}
Under the hypotheses of Theorem~\ref{mainthm}, fix $0<\alpha<\widetilde\alpha$. There exist $r_0>0$ and $C>0$, depending only on the data in that theorem and on $\alpha$, such that the following holds. For every $z\in\partial\Omega$ and $t_0\in[-T/2,0]$, there is a unique number $a(z,t_0)$, with $|a(z,t_0)|\leq C$, satisfying
\begin{equation}\label{bdryest}
|u(x,t)-a(z,t_0)\psi(x)|
\leq C\left(|x-z|^{\sigma+\alpha}+|t-t_0|^{\frac{\sigma+\alpha}{\sigma(2-\gamma)}}\right)
\end{equation}
on $(\overline\Omega\times[-T,0])\cap\overline{Q_{r_0}(z,t_0)}$. Here $\psi$ is the normalized first eigenfunction from Lemma~\ref{firsteigenfunction}, and $r_0^{\sigma(2-\gamma)}<T/4$. The estimate also holds on the full cylinder when $u$ and $\psi$ are extended by zero.
\end{Theorem}

\begin{proof}
Suppose that \eqref{bdryest} fails uniformly with the data fixed. After translations and rotations, take
\[
0\in\partial\Omega^{(k)},\qquad \nu_k(0)=e_n,\qquad
\Omega^{(k)}\subset B_C,\qquad L_k\in\mathcal L_*,
\]
Here $\Omega^{(k)}$ is a translated and rotated copy of $\Omega$, and $\nu_k$ is its inward unit normal.
\begin{align*}
\partial_tu_k&=u_k^\gamma L_ku_k
\quad\text{in }\Omega^{(k)}\times(-T/2,0],\\
L_k\psi_k&=-\lambda_{1,k}\psi_k\quad\text{in }\Omega^{(k)},\\
\|\psi_k\|_{L^2}&=1,\qquad
0<\lambda_{1,k}\leq C,\\
u_k&=\psi_k=0\quad\text{on }(\Omega^{(k)})^c\times[-T/2,0].
\end{align*}
Set $d^{(k)}(x)=\operatorname{dist}(x,(\Omega^{(k)})^c)$. Lemma~\ref{firsteigenfunction} and Theorem~\ref{thm1} give
\[
C_1(d^{(k)})^\sigma\leq u_k\leq C_2(d^{(k)})^\sigma,
\qquad
c(d^{(k)})^\sigma\leq\psi_k\leq C(d^{(k)})^\sigma,\qquad
\|u_k\|_{L^\infty}+\|\psi_k\|_{L^\infty}\leq C.
\]
Choose $r_0>0$, depending only on the uniform geometric bounds, with $r_0^{\sigma(2-\gamma)}<T/4$. Then
\begin{equation}\label{eq:projection-bounds}
0\leq u_k\leq C\psi_k,\qquad
cr^\sigma\leq\|\psi_k\|_{L^\infty(B_r)}\leq Cr^\sigma
\quad(0<r\leq r_0).
\end{equation}

\textbf{Step 1. }
For $0<r\leq r_0$, set
\[
a_k(r)=\frac{\int_{Q_r}u_k\psi_k\,dx\,dt}
{\int_{Q_r}\psi_k^2\,dx\,dt},
\qquad
\theta(r)=\sup_k\sup_{r\leq r'\leq r_0}
\frac{\|u_k-a_k(r')\psi_k\|_{L^\infty(Q_{r'})}}
{(r')^{\sigma+\alpha}}.
\]
The bounds $c(d^{(k)})^\sigma\leq\psi_k\leq C(d^{(k)})^\sigma$
and the uniform $C^{1,1}$ regularity of $\Omega^{(k)}$ give
\[
cr^{n+2\sigma+\sigma(2-\gamma)}
\leq\int_{Q_r}\psi_k^2\,dx\,dt
\leq Cr^{n+2\sigma+\sigma(2-\gamma)}.
\]
Moreover, since $0\leq u_k\leq C\psi_k$, we have
\[
0\leq\int_{Q_r}u_k\psi_k\,dx\,dt
\leq C\int_{Q_r}\psi_k^2\,dx\,dt.
\]
Thus $0\leq a_k(r)\leq C$. Combining this with
\eqref{eq:projection-bounds}, we obtain
\[
\|u_k-a_k(r)\psi_k\|_{L^\infty(Q_r)}\leq Cr^\sigma,
\]
and hence
\[
0\leq\theta(r)\leq Cr^{-\alpha}<\infty
\qquad\mbox{for }0<r\leq r_0.
\]
Finally, $\theta$ is nonincreasing: if $0<r_1\leq r_2\leq r_0$,
then $[r_2,r_0]\subset[r_1,r_0]$, so its definition gives
$\theta(r_1)\geq\theta(r_2)$.

For $0<r\leq r_0/2$,
\begin{align}\label{ak}
|a_k(2r)-a_k(r)|
&\leq\frac{\|u_k-a_k(2r)\psi_k\|_{L^\infty(Q_r)}
+\|u_k-a_k(r)\psi_k\|_{L^\infty(Q_r)}}
{\|\psi_k\|_{L^\infty(B_r)}}\leq C\theta(r)r^\alpha.
\end{align}
The same estimate gives
\[
|a_k(r')-a_k(r)|\leq C\theta(r)r^\alpha
\qquad\mbox{when }r\leq r'\leq2r\leq r_0.
\]
If $M:=\sup_{0<r\leq r_0}\theta(r)<\infty$, then, for $m>\ell$,
\[
|a_k(2^{-m}r)-a_k(2^{-\ell}r)|
\leq CM\sum_{i=\ell}^{m-1}(2^{-i}r)^\alpha
\leq CM(2^{-\ell}r)^\alpha.
\]
Consequently, $a_k(r)\longrightarrow a_k^*$ as $r\downarrow0$, $|a_k(r)-a_k^*|\leq CMr^\alpha$, and
\begin{align*}
\|u_k-a_k^*\psi_k\|_{L^\infty(Q_r)}
&\leq\|u_k-a_k(r)\psi_k\|_{L^\infty(Q_r)}
+|a_k(r)-a_k^*|\|\psi_k\|_{L^\infty(B_r)}\\
&\leq CMr^{\sigma+\alpha}.
\end{align*}
This contradicts the assumed failure. Hence
\[
\theta(r)\longrightarrow\infty\qquad \mbox{as }r\downarrow0.
\]

Choose $\rho_j\downarrow0$ with $\theta(\rho_j)\geq j$, and select $r_j\in[\rho_j,r_0]$ and $k_j\in\mathbb N$ such that
\[
\frac{\|u_{k_j}-a_{k_j}(r_j)\psi_{k_j}\|_{L^\infty(Q_{r_j})}}
{r_j^{\sigma+\alpha}}
\geq\frac12\theta(\rho_j).
\]
Since
\[
\frac12\theta(\rho_j)\leq Cr_j^{-\alpha},
\qquad
\theta(r_j)\leq\theta(\rho_j),
\]
we have $r_j\to0$, $\theta(r_j)\to\infty$, and
\[
\frac12\theta(r_j)\leq
\frac{\|u_{k_j}-a_{k_j}(r_j)\psi_{k_j}\|_{L^\infty(Q_{r_j})}}
{r_j^{\sigma+\alpha}}
\leq\theta(r_j).
\]

Since $r_j\to0$ and $\theta(r_j)\to\infty$, after passing to a
subsequence we may assume that $r_j\downarrow0$ and
$\theta(r_j)\geq1$ for every $j$.
Set
\[
\Omega_j=r_j^{-1}\Omega^{(k_j)},\qquad
T_j=\frac{T}{2r_j^{\sigma(2-\gamma)}},
\]
and define the rescaled functions
\[
v_j(x,t)=
\frac{u_{k_j}(r_jx,r_j^{\sigma(2-\gamma)}t)
-a_{k_j}(r_j)\psi_{k_j}(r_jx)}
{r_j^{\sigma+\alpha}\theta(r_j)}
\quad\text{for }(x,t)\in\R^n\times[-T_j,0].
\]
Then $T_j\to\infty$ and
\[
v_j=0\quad\text{on }\Omega_j^c\times[-T_j,0].
\]
For each $j$, the function $v_j$ is continuous and vanishes outside
the bounded domain $\Omega_j$. Thus, by dominated convergence,
$v_j$ satisfies \eqref{eq:tail-class} on $[-T_j,0]$. Moreover,
\begin{equation}\label{eq:normalized-error}
\frac12\leq\|v_j\|_{L^\infty(Q_1)}\leq1,\qquad
\int_{Q_1}v_j(x,t)\frac{\psi_{k_j}(r_jx)}{r_j^\sigma}\,dx\,dt=0.
\end{equation}

\textbf{Step 2.}
For $R=2^m\geq1$ and $Rr\leq r_0$, \eqref{ak} gives
\[
|a_k(Rr)-a_k(r)|
\leq C\sum_{i=0}^{m-1}\theta(2^ir)(2^ir)^\alpha
\leq C\theta(r)(Rr)^\alpha.
\]
Thus, for $Rr_j\leq r_0$,
\begin{align*}
\|v_j\|_{L^\infty(Q_R)}
&=\frac{\|u_{k_j}-a_{k_j}(r_j)\psi_{k_j}\|_{L^\infty(Q_{Rr_j})}}
{r_j^{\sigma+\alpha}\theta(r_j)}\\
&\leq\frac{(Rr_j)^{\sigma+\alpha}\theta(Rr_j)
+C\theta(r_j)(Rr_j)^\alpha(Rr_j)^\sigma}
{r_j^{\sigma+\alpha}\theta(r_j)}\\
&\leq CR^{\sigma+\alpha}.
\end{align*}
For arbitrary $1\leq R\leq r_0/(2r_j)$, choose $2^{m-1}<R\leq2^m\leq2R$. Then
\[
2^mr_j\leq r_0,\qquad
\|v_j\|_{L^\infty(Q_R)}
\leq\|v_j\|_{L^\infty(Q_{2^m})}
\leq CR^{\sigma+\alpha}.
\]
For $Rr_j\geq r_0/2$,
\[
\sup_{t\in[-T_j,0]}\|v_j(\cdot,t)\|_{L^\infty(B_R)}
\leq\frac{C}{r_j^{\sigma+\alpha}\theta(r_j)}
\leq Cr_j^{-\sigma-\alpha}
\leq CR^{\sigma+\alpha}.
\]
Fix $S>0$. For large $j$, $T_j>S$. With $\rho=2\max\{1,|x|,S^{1/[\sigma(2-\gamma)]}\}$, the preceding two bounds give
\[
|v_j(x,t)|\leq C\rho^{\sigma+\alpha}
\leq C_S(1+|x|)^{\sigma+\alpha}
\qquad(-S\leq t\leq0).
\]
Hence, for $R\geq1$,
\[
\sup_{t\in[-S,0]}\int_{|y|>R}
\frac{|v_j(y,t)|}{|y|^{n+2\sigma}}\,dy
\leq C_S\int_R^\infty r^{\alpha-\sigma-1}\,dr
\leq C_S R^{\alpha-\sigma}.
\]
In particular,
\begin{equation}\label{eq:blowup-tail}
\begin{split}
|v_j(x,t)|&\leq C_S(1+|x|)^{\sigma+\alpha}
\quad(x\in\R^n,\ -S\leq t\leq0),\\
\sup_{t\in[-S,0]}\int_{|y|>R}
\frac{|v_j(y,t)|}{|y|^{n+2\sigma}}\,dy
&\leq C_S R^{\alpha-\sigma}\quad(R\geq1).
\end{split}
\end{equation}

\textbf{Step 3. }
Set
\[
d_j(x)=\operatorname{dist}(x,\Omega_j^c)
=r_j^{-1}d^{(k_j)}(r_jx),\qquad d_j(x)\leq|x|.
\]
Homogeneity and $L_{k_j}\psi_{k_j}=-\lambda_{1,k_j}\psi_{k_j}$ give
\begin{align*}
\partial_tv_j
&=\left(\frac{u_{k_j}(r_jx,r_j^{\sigma(2-\gamma)}t)}{r_j^\sigma}\right)^\gamma
L_{k_j}v_j-
\left(\frac{u_{k_j}(r_jx,r_j^{\sigma(2-\gamma)}t)}{r_j^\sigma}\right)^\gamma
\frac{a_{k_j}(r_j)\lambda_{1,k_j}r_j^{2\sigma}\psi_{k_j}(r_jx)}
{r_j^{\sigma+\alpha}\theta(r_j)}.
\end{align*}
In $\Omega_j\times(-T_j,0]$,
\[
C_1^\gamma d_j^{\sigma\gamma}
\leq
\left(\frac{u_{k_j}(r_jx,r_j^{\sigma(2-\gamma)}t)}{r_j^\sigma}\right)^\gamma
\leq C_2^\gamma d_j^{\sigma\gamma},
\qquad
\psi_{k_j}(r_jx)\leq Cr_j^\sigma d_j(x)^\sigma,
\]
and the enlarged extremal class therefore gives
\[
\begin{cases}
\displaystyle
\partial_tv_j-d_j^{\sigma\gamma}M^+v_j
\leq C\frac{r_j^{2\sigma-\alpha}}{\theta(r_j)}
d_j^{\sigma(1+\gamma)},\\[6pt]
\displaystyle
\partial_tv_j-d_j^{\sigma\gamma}M^-v_j
\geq-C\frac{r_j^{2\sigma-\alpha}}{\theta(r_j)}
d_j^{\sigma(1+\gamma)}.
\end{cases}
\]
For every fixed $R$,
\[
\sup_{B_R}C\frac{r_j^{2\sigma-\alpha}}{\theta(r_j)}
d_j^{\sigma(1+\gamma)}
\leq C_R\frac{r_j^{2\sigma-\alpha}}{\theta(r_j)}
\longrightarrow0.
\]
The uniform $C^{1,1}$ geometry and $\nu_{k_j}(0)=e_n$ yield
\[
\Omega_j\longrightarrow\R^n_+,\qquad
d_j\longrightarrow(x_n)_+
\quad\text{locally uniformly}.
\]
By Corollary~\ref{cor:compactness} and \eqref{eq:blowup-tail}, after passing to a subsequence, 
$v_j\longrightarrow v$ locally uniformly on $\R^n\times(-\infty,0]$, $v=0$ on $\{x_n\leq0\}\times(-\infty,0]$, $\|v\|_{L^\infty(Q_R^+)}\leq CR^{\sigma+\alpha}$ for $R\geq1$, and
\[
\begin{cases}
\partial_tv-x_n^{\sigma\gamma}M^+v\leq0,\\
\partial_tv-x_n^{\sigma\gamma}M^-v\geq0,
\end{cases}
\quad\text{in }\R^n_+\times(-\infty,0].
\]
Theorem~\ref{liou} gives
\[
v=b(x_n)_+^\sigma\qquad\text{for some }b\in\R.
\]

\textbf{Step 4. }
By Lemma~\ref{firsteigenfunction},
\[
\left\|\frac{\psi_k}{(d^{(k)})^\sigma}\right\|_{C^\alpha(\overline{\Omega^{(k)}})}
\leq C,\qquad
c\leq\left(\frac{\psi_k}{(d^{(k)})^\sigma}\right)(0)\leq C.
\]
After passing to a further subsequence,
\[
\kappa_j:=
\left(\frac{\psi_{k_j}}{(d^{(k_j)})^\sigma}\right)(0)
\longrightarrow\kappa\in[c,C].
\]
For $x\in\Omega_j\cap B_R$,
\begin{align*}
\left|\frac{\psi_{k_j}(r_jx)}{r_j^\sigma}
-\kappa_jd_j(x)^\sigma\right|
&=d_j(x)^\sigma
\left|\frac{\psi_{k_j}(r_jx)}
{d^{(k_j)}(r_jx)^\sigma}-\kappa_j\right|\leq C d_j(x)^\sigma(r_j|x|)^\alpha\leq C_Rr_j^\alpha.
\end{align*}
Both terms vanish in $\Omega_j^c$. Therefore,
\begin{align*}
&\left\|\frac{\psi_{k_j}(r_j\,\cdot)}{r_j^\sigma}
-\kappa(x_n)_+^\sigma\right\|_{L^\infty(B_R)}\leq C_Rr_j^\alpha+C_R|\kappa_j-\kappa|
+C\|d_j-(x_n)_+\|_{L^\infty(B_R)}^\sigma\longrightarrow0.
\end{align*}
Passing to the limit in \eqref{eq:normalized-error},
\[
0=\lim_{j\to\infty}
\int_{Q_1}v_j(x,t)\frac{\psi_{k_j}(r_jx)}{r_j^\sigma}\,dx\,dt
=b\kappa\int_{-1}^0\int_{B_1^+}x_n^{2\sigma}\,dx\,dt.
\]
Thus
\[
b=0,\qquad
\frac12\leq\lim_{j\to\infty}\|v_j\|_{L^\infty(Q_1)}
=\|v\|_{L^\infty(Q_1)}=0,
\]
a contradiction.

The argument for bounded $\theta$ in Step~1 now gives, at every base point,
\begin{align*}
&0\leq a(z,t_0)=\lim_{r\downarrow0}
\frac{\int_{Q_r(z,t_0)}u\psi\,dx\,dt}
{\int_{Q_r(z,t_0)}\psi^2\,dx\,dt}\leq C,\\
&\|u-a(z,t_0)\psi\|_{L^\infty(Q_r(z,t_0))}
\leq Cr^{\sigma+\alpha}\qquad\mbox{for }0<r\leq r_0.
\end{align*}
For $(x,t)\in\overline{Q_{r_0}(z,t_0)}$, set
\[
r=\max\{|x-z|,|t-t_0|^{1/[\sigma(2-\gamma)]}\}.
\]
For $r>0$,
\[
|u(x,t)-a(z,t_0)\psi(x)|
\leq Cr^{\sigma+\alpha}
\leq C\left(|x-z|^{\sigma+\alpha}
+|t-t_0|^{(\sigma+\alpha)/[\sigma(2-\gamma)]}\right).
\]
For $r=0$, both sides vanish. Finally, if $a,\widehat a$ both satisfy \eqref{bdryest} at $(z,t_0)$, then, for $x_h=z+h\nu(z)$, we have $\psi(x_h)\geq ch^\sigma$ and
\begin{align*}
|a-\widehat a|
&\leq\frac{|u(x_h,t_0)-a\psi(x_h)|
+|u(x_h,t_0)-\widehat a\psi(x_h)|}{\psi(x_h)}\leq Ch^\alpha\longrightarrow0.
\end{align*}
Hence $a=\widehat a$.
\end{proof}

\subsection{H\"older regularity of the quotient}
\begin{proof}[Proof of Theorem~\ref{mainthm}]
We combine the boundary expansion in Theorem~\ref{point} with the interior H\"older estimate in Theorem~\ref{thm:interior-estimates} to prove H\"older continuity of $u/\psi$. The desired estimate then follows from the regularity of $\psi/d^\sigma$.

Choose an interior H\"older exponent $\beta$ from Theorem~\ref{thm:interior-estimates}, with the ellipticity bounds used in Steps~2 and~3 below, decreasing it if necessary so that $\beta\in(0,\sigma)$, and fix
\[
0<\alpha<\min\left\{\widetilde\alpha,\frac{2-\gamma}{2}\beta\right\}.
\]
Then
\[
0<\alpha<\beta<\sigma,\qquad
\frac{\alpha}{\sigma(2-\gamma)}<\frac{\beta}{2\sigma}.
\]

\textbf{Step 1. H\"older continuity of the boundary coefficient.}
Let $(z_1,t_1),(z_2,t_2)\in\partial\Omega\times[-T/2,0]$, $t_1\leq t_2$ and 
$h=|z_1-z_2|+|t_1-t_2|^{1/[\sigma(2-\gamma)]}$.

Fix $A>1$ and choose $h_0>0$ so small that
$(A+1)h_0<r_0$ and $Ah_0$ is smaller than the tubular radius
of $\partial\Omega$. For $0<h<h_0$, take
\[
X=z_1+Ah\nu(z_1),
\]
where $\nu(z_1)$ is the inward unit normal.
Then $X\in\Omega$ and $z_1$ is its nearest boundary point, so
\[
d(X)=|X-z_1|=Ah.
\]
Moreover, the definition of $h$ gives $|t_2-t_1|\leq h^{\sigma(2-\gamma)}$ and 
\[
|X-z_2|\leq|X-z_1|+|z_1-z_2|\leq(A+1)h. 
\]
The tubular geometry and $(A+1)h_0<r_0$ give
\[
(X,t_1)\in Q_{r_0}(z_1,t_1)\cap Q_{r_0}(z_2,t_2).
\]
By Theorem~\ref{point}, for $i=1,2$,
\begin{align*}
|u(X,t_1)-a(z_i,t_i)\psi(X)|
&\leq C\left(|X-z_i|^{\sigma+\alpha}
+|t_i-t_1|^{(\sigma+\alpha)/[\sigma(2-\gamma)]}\right)\leq Ch^{\sigma+\alpha},
\end{align*}
and
\[
\psi(X)\geq c\,d(X)^\sigma\geq ch^\sigma.
\]
Thus
\[
|a(z_1,t_1)-a(z_2,t_2)|
\leq
\frac{\sum_{i=1}^2|u(X,t_1)-a(z_i,t_i)\psi(X)|}{\psi(X)}
\leq Ch^\alpha.
\]
For $h\geq h_0$,
\[
|a(z_1,t_1)-a(z_2,t_2)|
\leq2\|a\|_{L^\infty}
\leq2h_0^{-\alpha}\|a\|_{L^\infty}h^\alpha.
\]
For $h=0$, the difference is zero. Hence
\begin{equation}\label{eq:coefficient-holder}
|a(z_1,t_1)-a(z_2,t_2)|\leq Ch^\alpha.
\end{equation}

For $x\in\Omega$ near $\partial\Omega$, let $z_x$ be its nearest boundary point. Then
\[
|x-z_x|=d(x),\qquad
|u(x,t)-a(z_x,t)\psi(x)|\leq Cd(x)^{\sigma+\alpha},
\qquad
\psi(x)\geq cd(x)^\sigma,
\]
and therefore
\begin{equation}\label{eq:quotient-trace}
\left|\frac{u(x,t)}{\psi(x)}-a(z_x,t)\right|\leq Cd(x)^\alpha.
\end{equation}
For $(z,t_0)\in\partial\Omega\times[-T/2,0]$,
\[
|z_x-z|\leq d(x)+|x-z|\leq2|x-z|,
\]
and
\begin{align*}
\left|\frac{u(x,t)}{\psi(x)}-a(z,t_0)\right|
&\leq Cd(x)^\alpha
+C\left(|z_x-z|+|t-t_0|^{1/[\sigma(2-\gamma)]}\right)^\alpha\longrightarrow0
\end{align*}
as $(x,t)\to(z,t_0)$. Define the boundary extension by
\[
\left.\frac{u}{\psi}\right|_{\partial\Omega\times[-T/2,0]}=a.
\]

\textbf{Step 2. An interior estimate for the error.}
Choose $\delta_0\in(0,\min\{1,r_0/4\})$ within the tubular radius. Fix $t\in[-T/2,0]$, $x\in\Omega$, 
$\delta=d(x)<\delta_0$,  $z=z_x$, $a=a(z,t)$,
and set
\[
w(y,\tau)=u(y,\tau)-a\psi(y),\qquad
\widetilde w(y,\tau)=w(x+\delta y,t+\delta^{\sigma(2-\gamma)}\tau).
\]
For $y\in B_{1/2}$ and $\tau\in[-1,0]$,
\[
\frac{\delta}{2}\leq d(x+\delta y)\leq\frac{3\delta}{2},\qquad
t+\delta^{\sigma(2-\gamma)}\tau
\geq-T/2-\delta_0^{\sigma(2-\gamma)}>-T.
\]
Moreover,
\begin{align*}
&\partial_\tau\widetilde w
=\left(\frac{u(x+\delta y,t+\delta^{\sigma(2-\gamma)}\tau)}
{\delta^\sigma}\right)^\gamma L\widetilde w-a\lambda_1\delta^{\sigma(2-\gamma)}
u(x+\delta y,t+\delta^{\sigma(2-\gamma)}\tau)^\gamma
\psi(x+\delta y),\\
&0<c\leq
\left(\frac{u(x+\delta y,t+\delta^{\sigma(2-\gamma)}\tau)}
{\delta^\sigma}\right)^\gamma
\leq C,\\
&\left|a\lambda_1\delta^{\sigma(2-\gamma)}
u(x+\delta y,t+\delta^{\sigma(2-\gamma)}\tau)^\gamma
\psi(x+\delta y)\right|
\leq C\delta^{\sigma(2-\gamma)+\sigma\gamma+\sigma}=C\delta^{3\sigma}.
\end{align*}

For $|y|\leq r_0/(2\delta)$,
\[
|x+\delta y-z|\leq\delta(1+|y|)\leq\delta+r_0/2<r_0.
\]
Theorem~\ref{point} gives, for $-1\leq\tau\leq0$,
\begin{align*}
|\widetilde w(y,\tau)|
&\leq C\left(|x+\delta y-z|^{\sigma+\alpha}
+|\delta^{\sigma(2-\gamma)}\tau|^{(\sigma+\alpha)/[\sigma(2-\gamma)]}\right)\leq C\delta^{\sigma+\alpha}(1+|y|)^{\sigma+\alpha}.
\end{align*}
Also,
\[
\|\widetilde w\|_{L^\infty(\R^n\times[-1,0])}\leq C.
\]
Since $\alpha<\sigma$,
\begin{align*}
&\sup_{\tau\in[-1,0]}
\int_{\R^n}\frac{|\widetilde w(y,\tau)|}{1+|y|^{n+2\sigma}}\,dy\\
&\qquad\leq C\delta^{\sigma+\alpha}
\int_{\R^n}\frac{(1+|y|)^{\sigma+\alpha}}{1+|y|^{n+2\sigma}}\,dy
+C\int_{|y|>r_0/(2\delta)}\frac{dy}{|y|^{n+2\sigma}}\\
&\qquad\leq C\delta^{\sigma+\alpha}+C\delta^{2\sigma}\\
&\qquad\leq C\delta^{\sigma+\alpha}.
\end{align*}
Consequently,
\begin{equation}\label{eq:interior-error-tail}
\sup_{\tau\in[-1,0]}\left[
\|\widetilde w(\cdot,\tau)\|_{L^\infty(B_{1/2})}
+\int_{\R^n}\frac{|\widetilde w(y,\tau)|}{1+|y|^{n+2\sigma}}\,dy
\right]\leq C\delta^{\sigma+\alpha}.
\end{equation}
After division by $\delta^{\sigma+\alpha}$, the forcing is bounded by
\[
C\delta^{3\sigma-(\sigma+\alpha)}
=C\delta^{2\sigma-\alpha}\leq C.
\]
Applying \eqref{eq:interior-holder} in Theorem~\ref{thm:interior-estimates} on a fixed finite covering of interior cylinders, with the coefficient absorbed into the ellipticity constants, and using \eqref{eq:interior-error-tail}, we obtain a fixed $c\in(0,1)$ such that
\begin{equation}\label{eq:interior-error-holder}
\|\widetilde w\|_{C^{\beta,\beta/(2\sigma)}_{x,t}
(\overline B_{1/4}\times[-c,0])}
\leq C\delta^{\sigma+\alpha}.
\end{equation}

\textbf{Step 3. The spatial estimate.}
Let $x,y\in\Omega$ and $t\in[-T/2,0]$, and set $h=|x-y|$.

If $\delta=d(x)<\delta_0$ and $h\leq\delta/8$, then \eqref{eq:interior-error-holder} gives
\[
|w(x,t)-w(y,t)|
\leq C\delta^{\sigma+\alpha}\left(\frac{h}{\delta}\right)^\beta,
\qquad
|w(y,t)|\leq C\delta^{\sigma+\alpha}.
\]
Furthermore, $c\delta^\sigma\leq\psi(x)\leq C\delta^\sigma$, $c\delta^\sigma\leq\psi(y)\leq C\delta^\sigma$ and
\[
|\psi(x)-\psi(y)|
\leq Ch^\sigma
=C\delta^\sigma\left(\frac{h}{\delta}\right)^\sigma
\leq C\delta^\sigma\left(\frac{h}{\delta}\right)^\beta.
\]
Since $u/\psi=a+w/\psi$,
\begin{align*}
\left|\frac{u(x,t)}{\psi(x)}-\frac{u(y,t)}{\psi(y)}\right|
&\leq\frac{|w(x,t)-w(y,t)|}{\psi(x)}
+\frac{|w(y,t)|\,|\psi(x)-\psi(y)|}{\psi(x)\psi(y)}\\
&\leq C\delta^\alpha\left(\frac{h}{\delta}\right)^\beta\\
&=Ch^\alpha\left(\frac{h}{\delta}\right)^{\beta-\alpha}\\
&\leq Ch^\alpha.
\end{align*}

If $\delta=d(x)<\delta_0$ and $h>\delta/8$, then
\[
d(x)\leq8h,\qquad
d(y)\leq d(x)+h\leq9h.
\]
For $h<\delta_0/9$,
\[
|z_x-z_y|\leq d(x)+h+d(y)\leq18h.
\]
Thus \eqref{eq:quotient-trace} and \eqref{eq:coefficient-holder} imply
\begin{align*}
\left|\frac{u(x,t)}{\psi(x)}-\frac{u(y,t)}{\psi(y)}\right|
&\leq
\left|\frac{u(x,t)}{\psi(x)}-a(z_x,t)\right|
+|a(z_x,t)-a(z_y,t)|+\left|a(z_y,t)-\frac{u(y,t)}{\psi(y)}\right|\\
&\leq C\bigl(d(x)^\alpha+|z_x-z_y|^\alpha+d(y)^\alpha\bigr)\\
&\leq Ch^\alpha.
\end{align*}
For $h\geq\delta_0/9$,
\[
\left|\frac{u(x,t)}{\psi(x)}-\frac{u(y,t)}{\psi(y)}\right|
\leq2\left\|\frac{u}{\psi}\right\|_{L^\infty}
\leq Ch^\alpha.
\]

For $d(x)\geq\delta_0$ and $h<\delta_0/8$,
\[
d(y)\geq7\delta_0/8,\qquad
\inf_{\{d\geq\delta_0/2\}}\psi>0.
\]
The H\"older estimate \eqref{eq:interior-holder} in Theorem~\ref{thm:interior-estimates}, applied to $u$ on interior cylinders, together with Lemma~\ref{firsteigenfunction}, gives
\[
\left\|\frac{u}{\psi}\right\|_{C^{\beta,\beta/(2\sigma)}_{x,t}
(\{d\geq\delta_0/2\}\times[-T/2,0])}\leq C,
\]
and hence
\[
\left|\frac{u(x,t)}{\psi(x)}-\frac{u(y,t)}{\psi(y)}\right|
\leq Ch^\beta\leq Ch^\alpha.
\]
For $h\geq\delta_0/8$, use the uniform bound as above. By continuity,
\begin{equation}\label{space}
\left|\frac{u(x,t)}{\psi(x)}-\frac{u(y,t)}{\psi(y)}\right|
\leq C|x-y|^\alpha
\quad\mbox{for }x,y\in\overline\Omega,\ -T/2\leq t\leq0.
\end{equation}

\textbf{Step 4. The temporal estimate.}
Fix $-T/2\leq t_1<t_2\leq0$ and $x\in\Omega$ with $\delta=d(x)<\delta_0$. Use Step~2 with terminal time $t_2$ and $a=a(z_x,t_2)$. Then
\[
u(x,t_2)-u(x,t_1)=w(x,t_2)-w(x,t_1).
\]
If $t_2-t_1\leq c\delta^{\sigma(2-\gamma)}$, \eqref{eq:interior-error-holder} gives
\begin{align*}
\left|\frac{u(x,t_2)-u(x,t_1)}{\psi(x)}\right|
&\leq C\delta^\alpha
\left(\frac{t_2-t_1}{\delta^{\sigma(2-\gamma)}}\right)^{\beta/(2\sigma)}\\
&=C|t_2-t_1|^{\alpha/[\sigma(2-\gamma)]}
\left(\frac{t_2-t_1}{\delta^{\sigma(2-\gamma)}}\right)^{
\beta/(2\sigma)-\alpha/[\sigma(2-\gamma)]}\\
&\leq C|t_2-t_1|^{\alpha/[\sigma(2-\gamma)]}.
\end{align*}
Here
\[
\frac{t_2-t_1}{\delta^{\sigma(2-\gamma)}}\leq c<1,\qquad
\frac{\beta}{2\sigma}-\frac{\alpha}{\sigma(2-\gamma)}>0.
\]
If $t_2-t_1>c\delta^{\sigma(2-\gamma)}$, then
\begin{align*}
\left|\frac{u(x,t_2)}{\psi(x)}-\frac{u(x,t_1)}{\psi(x)}\right|
&\leq
\sum_{i=1}^2\left|\frac{u(x,t_i)}{\psi(x)}-a(z_x,t_i)\right|
+|a(z_x,t_2)-a(z_x,t_1)|\\
&\leq C\delta^\alpha
+C|t_2-t_1|^{\alpha/[\sigma(2-\gamma)]}\\
&\leq C|t_2-t_1|^{\alpha/[\sigma(2-\gamma)]}.
\end{align*}
For $d(x)\geq\delta_0$, the interior estimate in Step~3 gives, when $t_2-t_1\leq1$,
\[
\left|\frac{u(x,t_2)}{\psi(x)}-\frac{u(x,t_1)}{\psi(x)}\right|
\leq C|t_2-t_1|^{\beta/(2\sigma)}
\leq C|t_2-t_1|^{\alpha/[\sigma(2-\gamma)]}.
\]
When $t_2-t_1>1$, the same bound follows from $\|u/\psi\|_{L^\infty}\leq C$. Thus
\begin{equation}\label{time}
\left|\frac{u(x,t_2)}{\psi(x)}-\frac{u(x,t_1)}{\psi(x)}\right|
\leq C|t_2-t_1|^{\alpha/[\sigma(2-\gamma)]}
\quad\mbox{for }x\in\overline\Omega.
\end{equation}

\textbf{Step 5. Returning to the distance quotient.}
For $x,y\in\overline\Omega$ and $s,t\in[-T/2,0]$, \eqref{space} and \eqref{time} give
\begin{align*}
\left|\frac{u(x,t)}{\psi(x)}-\frac{u(y,s)}{\psi(y)}\right|
&\leq
\left|\frac{u(x,t)}{\psi(x)}-\frac{u(y,t)}{\psi(y)}\right|
+\left|\frac{u(y,t)}{\psi(y)}-\frac{u(y,s)}{\psi(y)}\right|\\
&\leq C\left(|x-y|^\alpha+|t-s|^{\alpha/[\sigma(2-\gamma)]}\right).
\end{align*}
Consequently,
\[
\left\|\frac{u}{\psi}\right\|_{C^{\alpha,\alpha/[\sigma(2-\gamma)]}_{x,t}
(\overline\Omega\times[-T/2,0])}\leq C.
\]
Lemma~\ref{firsteigenfunction} gives
\[
\left\|\frac{\psi}{d^\sigma}\right\|_{C^\alpha(\overline\Omega)}\leq C.
\]
Using
\[
\frac{u}{d^\sigma}=\frac{u}{\psi}\frac{\psi}{d^\sigma},
\qquad
\left(\frac{u}{d^\sigma}\right)(z,t)
:=a(z,t)\left(\frac{\psi}{d^\sigma}\right)(z)
\quad\mbox{for }z\in\partial\Omega,
\]
we obtain
\begin{align*}
&\left\|\frac{u}{d^\sigma}\right\|_{C^{\alpha,\alpha/[\sigma(2-\gamma)]}_{x,t}
(\overline\Omega\times[-T/2,0])}\leq C
\left\|\frac{u}{\psi}\right\|_{C^{\alpha,\alpha/[\sigma(2-\gamma)]}_{x,t}
(\overline\Omega\times[-T/2,0])}
\left\|\frac{\psi}{d^\sigma}\right\|_{C^\alpha(\overline\Omega)}\leq C.
\end{align*}
\end{proof}

\subsection{Application to the fractional porous medium equation}
\begin{proof}[Proof of Theorem~\ref{thm:fpme}]
The case $v_0=0$ is immediate. Otherwise, by the regularity at positive times recalled in the introduction, $u=v^m$ is a viscosity solution of \eqref{eq:tfpme}, with $\gamma=1-1/m$ and $L=-m(-\Delta)^\sigma$. Its extension by zero is continuous and satisfies \eqref{eq:tail-class} on compact time intervals away from zero.

Fix $t_0>0$. For $a\geq t_0/2$, set
\[
U(x,\tau)=a^{1/\gamma}u(x,a(\tau+2)),\qquad -1\leq\tau\leq0.
\]
By \eqref{eq:tfpme} and \eqref{eq:dual-global-harnack},
\[
U_\tau=U^\gamma LU\quad\text{in }\Omega\times(-1,0],\qquad
cd^\sigma\leq U(\cdot,-1)\leq Cd^\sigma,
\]
with $c,C>0$ independent of $a$. Theorem~\ref{mainthm} yields
\[
\left\|\frac{U}{d^\sigma}\right\|_{C_{x,t}^{\alpha,\beta}
(\overline\Omega\times[-1/2,0])}\leq C,\qquad
\beta=\frac{\alpha}{\sigma(2-\gamma)}
=\frac{\alpha m}{\sigma(m+1)}.
\]
Rescaling and taking $a=2t/3$, for $t\geq t_0$ we obtain
\begin{align*}
\left\|\frac{u(\cdot,t)}{d^\sigma}\right\|_{C^\alpha(\overline\Omega)}
&\leq Ct^{-1/\gamma},\\
\left[\frac{u}{d^\sigma}\right]_{C_t^\beta
(\overline\Omega\times[t,4t/3])}
&\leq Ct^{-1/\gamma-\beta}.
\end{align*}
For $t'>4t/3$,
\[
t'-t>t_0/3,\qquad
\sup_{x\in\Omega}
\frac{|u(x,t')-u(x,t)|}{d(x)^\sigma|t'-t|^\beta}
\leq 2Ct_0^{-1/\gamma}(t_0/3)^{-\beta}.
\]
The triangle inequality then gives the H\"older estimate for $u/d^{\sigma}$, uniformly in $t_1>t_0$. The asserted estimate for $v/d^{\sigma/m}$ then follows from the preceding estimate and \eqref{eq:dual-global-harnack}.
\end{proof}
\appendix
\section{Proofs of the comparison principles}\label{app:comparison}
We prove the comparison principles for exterior values satisfying \eqref{eq:tail-class}. Continuity is required only on the closure of the equation domain.

\subsection{The nonlocal comparison at a doubled maximum}\label{appa1}
Let $D$ be a bounded open set, and let $u,v$ be continuous on $\overline D\times[-T,0]$ and satisfy \eqref{eq:tail-class}. Suppose
\[
\begin{split}
m&:=\max_{\overline D\times[-T,0]}(u-v)>0,\\
u&\leq v\quad\text{on }(\overline D\times\{-T\})\cup(\partial D\times[-T,0]),\\
u(\cdot,t)&\leq v(\cdot,t)\quad\text{a.e. in }D^c,\  t\in[-T,0].
\end{split}
\]
In particular,
\[
\bigl\{(x,t)\in\overline D\times[-T,0]:(u-v)(x,t)=m\bigr\}
\subset D\times(-T,0].
\]

For $\epsilon,\tau>0$, define
\[
m_{\epsilon,\tau}:=
\max_{\substack{x,y\in\overline D\\t,s\in[-T,0]}}
\left[u(x,t)-v(y,s)-\frac{|x-y|^2}{2\epsilon}
-\frac{|t-s|^2}{2\tau}\right].
\]
Choose a maximizing point $(x_\epsilon,t_\tau,y_\epsilon,s_\tau)$, so that
\[
m_{\epsilon,\tau}
=u(x_\epsilon,t_\tau)-v(y_\epsilon,s_\tau)
-\frac{|x_\epsilon-y_\epsilon|^2}{2\epsilon}
-\frac{|t_\tau-s_\tau|^2}{2\tau}.
\]
All four coordinates depend on both $\epsilon$ and $\tau$; the subscripts are used for brevity. By compactness and the doubling lemma, after passing to a subsequence as $(\epsilon,\tau)\to(0,0)$, we have
\begin{equation}\label{eq:double-limits}
\begin{split}
(x_\epsilon,t_\tau),(y_\epsilon,s_\tau)&\longrightarrow(x_0,t_0)\in D\times(-T,0],\\
\frac{|x_\epsilon-y_\epsilon|^2}{\epsilon}+\frac{|t_\tau-s_\tau|^2}{\tau}&\longrightarrow0,\\
m_{\epsilon,\tau}&\longrightarrow m=(u-v)(x_0,t_0).
\end{split}
\end{equation}
See Lemma~3.1 of \cite{CIL1992}. Define
\begin{align*}
\varphi_u(x,t)&=v(y_\epsilon,s_\tau)
+\frac{|x-y_\epsilon|^2}{2\epsilon}
+\frac{|t-s_\tau|^2}{2\tau}+m_{\epsilon,\tau},\\
\varphi_v(y,s)&=u(x_\epsilon,t_\tau)
-\frac{|x_\epsilon-y|^2}{2\epsilon}
-\frac{|t_\tau-s|^2}{2\tau}-m_{\epsilon,\tau}.
\end{align*}
Then
\begin{align*}
\varphi_u&\geq u,\qquad \varphi_v\leq v \quad\text{on }\overline D\times[-T,0],\\
\varphi_u(x_\epsilon,t_\tau)&=u(x_\epsilon,t_\tau), \qquad \varphi_v(y_\epsilon,s_\tau)=v(y_\epsilon,s_\tau),\\
\partial_t\varphi_u(x_\epsilon,t_\tau)&=\partial_s\varphi_v(y_\epsilon,s_\tau)=\frac{t_\tau-s_\tau}{\tau}.
\end{align*}
Both test functions belong to $C^\infty(\R^n\times\R)$, so these identities also apply when $t_\tau=0$ or $s_\tau=0$, with touching from the past.

Let $c\in C_{\mathrm{loc}}^{0,1}(D)$, with $c>0$ in $D$, and consider
\[
\int_{\R^n}\delta f(x,c(x)z)
\frac{\mu(z/|z|)}{|z|^{n+2\sigma}}\,dz.
\]
The case $c\equiv1$ is included. Write
\[
c_x=c(x_\epsilon),\qquad c_y=c(y_\epsilon),
\qquad c_x,c_y\longrightarrow c(x_0)>0.
\]
Fix $h>0$ such that, for all sufficiently small parameters along the chosen subsequence,
\[
\overline{B_{2c_xh}(x_\epsilon)}
\cup\overline{B_{2c_yh}(y_\epsilon)}\subset D,
\qquad t_\tau,s_\tau>-T.
\]
Set
\[
\ell=\frac12\min\{t_\tau+T,s_\tau+T\}>0.
\]
For each $0<\rho<h$, the patched functions $U_\rho=u_{\varphi_u}$ and $V_\rho=v_{\varphi_v}$ in Definition~\ref{defn:vis} are given explicitly by
\[
U_\rho(x,t):=
\begin{cases}
\varphi_u(x,t), & (x,t)\in B_{c_x\rho}(x_\epsilon)\times(t_\tau-\ell,t_\tau+\ell),\\
u(x,t), & (x,t)\in \R^n\times[-T,0]\setminus(B_{c_x\rho}(x_\epsilon)\times(t_\tau-\ell,t_\tau+\ell)),
\end{cases}
\]
and
\[
V_\rho(y,s):=
\begin{cases}
\varphi_v(y,s), & (y,s)\in B_{c_y\rho}(y_\epsilon)\times(s_\tau-\ell,s_\tau+\ell),\\
v(y,s), & (y,s)\in \R^n\times[-T,0]\setminus(B_{c_y\rho}(y_\epsilon)\times(s_\tau-\ell,s_\tau+\ell)).
\end{cases}
\]
The spatial patch radii are $c_x\rho$ and $c_y\rho$, so both patches depend on $\rho$. All the estimates below are uniform over the angular densities defining $\mathcal L_*$. Following the nonlocal framework in~\cite{BarlesImbert2008}, we split the integral into small, intermediate,
and large jumps.

\textbf{Small jumps.} For $|z|<\rho$,
\[
\delta U_\rho(x_\epsilon,c_xz,t_\tau)
-\delta V_\rho(y_\epsilon,c_yz,s_\tau)
=\frac{c_x^2+c_y^2}{2\epsilon}|z|^2.
\]
Consequently,
\[
\begin{aligned}
&\int_{|z|<\rho}
\bigl[\delta U_\rho(x_\epsilon,c_xz,t_\tau)
-\delta V_\rho(y_\epsilon,c_yz,s_\tau)\bigr]
\frac{\mu(z/|z|)}{|z|^{n+2\sigma}}\,dz\leq\frac{C}{\epsilon}\int_0^\rho r^{1-2\sigma}\,dr
\leq\frac{C\rho^{2-2\sigma}}{\epsilon}.
\end{aligned}
\]

\textbf{Intermediate jumps.} For $\rho\leq|z|<h$, the points
$x_\epsilon\pm c_xz$ and $y_\epsilon\pm c_yz$ belong to $D$. Maximality gives, for each choice of sign,
\[
\begin{aligned}
&u(x_\epsilon\pm c_xz,t_\tau)-v(y_\epsilon\pm c_yz,s_\tau)
-u(x_\epsilon,t_\tau)+v(y_\epsilon,s_\tau)\\
&\qquad\leq
\frac{|x_\epsilon-y_\epsilon\pm(c_x-c_y)z|^2
-|x_\epsilon-y_\epsilon|^2}{2\epsilon}.
\end{aligned}
\]
Averaging the two inequalities yields
\[
\delta u(x_\epsilon,c_xz,t_\tau)
-\delta v(y_\epsilon,c_yz,s_\tau)
\leq\frac{|c_x-c_y|^2|z|^2}{2\epsilon}.
\]
Since $|c_x-c_y|\leq C|x_\epsilon-y_\epsilon|$, we obtain
\[
\begin{aligned}
\int_{\rho\leq|z|<h}
\bigl[\delta u(x_\epsilon,c_xz,t_\tau)
-\delta v(y_\epsilon,c_yz,s_\tau)\bigr]
\frac{\mu(z/|z|)}{|z|^{n+2\sigma}}\,dz
&\leq C h^{2-2\sigma}\frac{|c_x-c_y|^2}{\epsilon}\\
& \leq C h^{2-2\sigma}\frac{|x_\epsilon-y_\epsilon|^2}{\epsilon}.
\end{aligned}
\]
For $c\equiv1$, the first upper bound is zero.

\textbf{Large jumps.} Fix $h$ as above. The weighted time continuity gives
\[
\begin{aligned}
&\int_{|z|\geq h}
\frac{|u(x_\epsilon\pm c_xz,t_\tau)
-u(x_\epsilon\pm c_xz,t_0)|}{|z|^{n+2\sigma}}\,dz\\
&\qquad\leq C_h
\int_{\R^n}\frac{|u(y,t_\tau)-u(y,t_0)|}{1+|y|^{n+2\sigma}}\,dy
\longrightarrow0.
\end{aligned}
\]
For the fixed function $u(\cdot,t_0)$, approximation in weighted $L^1$ by smooth compactly supported functions gives
\[
\int_{|z|\geq h}
\frac{|u(x_\epsilon\pm c_xz,t_0)
-u(x_0\pm c(x_0)z,t_0)|}{|z|^{n+2\sigma}}\,dz
\longrightarrow0.
\]
The constants in these changes of variables are uniform, since
$x_\epsilon\to x_0$ and $c_x\to c(x_0)>0$. Combining these estimates and applying the same argument to $v$, we obtain
\[
\begin{aligned}
&\int_{|z|\geq h}
\frac{|u(x_\epsilon\pm c_xz,t_\tau)
-u(x_0\pm c(x_0)z,t_0)|}{|z|^{n+2\sigma}}\,dz\longrightarrow0,\\
&\int_{|z|\geq h}
\frac{|v(y_\epsilon\pm c_yz,s_\tau)
-v(x_0\pm c(x_0)z,t_0)|}{|z|^{n+2\sigma}}\,dz\longrightarrow0.
\end{aligned}
\]
Also,
\[
u(x_\epsilon,t_\tau)\longrightarrow u(x_0,t_0),
\qquad
v(y_\epsilon,s_\tau)\longrightarrow v(x_0,t_0).
\]
Since $|\mu|\leq\Lambda$, these convergences imply
\[
\sup_\mu\Bigg|\int_{|z|\geq h}\Bigl[\delta u(x_\epsilon,c_xz,t_\tau)-\delta v(y_\epsilon,c_yz,s_\tau)-\delta(u-v)(x_0,c(x_0)z,t_0)\Bigr]\frac{\mu(z/|z|)}{|z|^{n+2\sigma}}\,dz\Bigg|\longrightarrow0.
\]

Maximality and the exterior ordering give
\[
(u-v)(y,t_0)\leq m\quad\text{for a.e. }y\in\R^n,
\qquad (u-v)(x_0,t_0)=m.
\]
Hence
\[
\begin{aligned}
&\delta(u-v)(x_0,c(x_0)z,t_0)=\frac{(u-v)(x_0+c(x_0)z,t_0)+(u-v)(x_0-c(x_0)z,t_0)}2-m
\leq0
\end{aligned}
\]
for almost every $z$. Therefore,
\[
\limsup_{(\epsilon,\tau)\to(0,0)}
\sup_\mu\int_{|z|\geq h}
\bigl[\delta u(x_\epsilon,c_xz,t_\tau)
-\delta v(y_\epsilon,c_yz,s_\tau)\bigr]
\frac{\mu(z/|z|)}{|z|^{n+2\sigma}}\,dz\leq0.
\]

Combining the estimates on these three regions gives, along the subsequence in \eqref{eq:double-limits},
\begin{equation}\label{eq:jump-comparison}
\begin{aligned}
\limsup_{(\epsilon,\tau)\to(0,0)}
\limsup_{\rho\downarrow0}\sup_\mu
\Bigg[&
\int_{\R^n}\delta U_\rho(x_\epsilon,c_xz,t_\tau)
\frac{\mu(z/|z|)}{|z|^{n+2\sigma}}\,dz\\
&-
\int_{\R^n}\delta V_\rho(y_\epsilon,c_yz,s_\tau)
\frac{\mu(z/|z|)}{|z|^{n+2\sigma}}\,dz
\Bigg]\leq0.
\end{aligned}
\end{equation}

\subsection{Conditional comparison}
\begin{proof}[Proof of Proposition~\ref{conditional}]
For $\eta>0$, set
\[
v^\eta(x,t)=v(x,t)+\eta(t+T+1).
\]
Since the perturbation is constant in space, it leaves the nonlocal operator unchanged. If $\varphi$ touches $v^\eta$ from below, then $\varphi-\eta(t+T+1)$ touches $v$ from below on the same cylinder. For the corresponding patches, the assumption $Lv\leq0$ gives
\[
L(v^\eta)_\varphi=L v_{\varphi-\eta(t+T+1)}\leq0.
\]
Since $v^\eta\geq v\geq0$, we obtain
\[
\partial_t\varphi
\geq v^\gamma L(v^\eta)_\varphi+\eta
\geq (v^\eta)^\gamma L(v^\eta)_\varphi+\eta.
\]
Thus $v^\eta$ is a strict supersolution, including at the terminal time. The initial and exterior inequalities become strict.

Suppose $u-v^\eta$ has a positive maximum. Apply the doubling construction in Section \ref{appa1} to $u,v^\eta$. For sufficiently small $\epsilon,\tau$, set
\[
p=u(x_\epsilon,t_\tau)>q=v^\eta(y_\epsilon,s_\tau)\geq0.
\]
Let $\varphi_v$ be the lower test for $v^\eta$ at $(y_\epsilon,s_\tau)$. Then
\[
\widetilde{\varphi}_v(y,s):=\varphi_v(y,s)-\eta(s+T+1)
\]
touches $v$ from below at the same point. For the corresponding patch $v_{\widetilde{\varphi}_v,\rho}$, we have
\[
\begin{split}
V_\rho(\cdot,s_\tau)&=v_{\widetilde{\varphi}_v,\rho}(\cdot,s_\tau)+\eta(s_\tau+T+1),\\
LV_\rho(y_\epsilon,s_\tau)&=L v_{\widetilde{\varphi}_v,\rho}(y_\epsilon,s_\tau)\leq0.
\end{split}
\]
The equality uses the fact that $L$ annihilates spatial constants, and the inequality follows from $Lv(\cdot,s_\tau)\leq0$ in the viscosity sense.
The viscosity inequalities and equality of the time derivatives yield
\begin{align*}
\eta&\leq p^\gamma LU_\rho-q^\gamma LV_\rho\\
&=p^\gamma(LU_\rho-LV_\rho)+(p^\gamma-q^\gamma)LV_\rho\\
&\leq p^\gamma(LU_\rho-LV_\rho).
\end{align*}
The last inequality uses $p>q\geq0$ and $LV_\rho\leq0$. The numbers $p$ are uniformly bounded because the contact points lie in the compact closure of the equation domain. Estimate~\eqref{eq:jump-comparison}, with $c\equiv1$, gives a contradiction. Hence $u\leq v^\eta$, and letting $\eta\downarrow0$ proves the result.
\end{proof}

We next justify the alternative sign condition in Remark~\ref{remcomparison}. We give the evaluation argument explicitly because upper touching and the elliptic inequality $Lu\leq0$ have opposite test directions.

\begin{proof}[Proof of Remark~\ref{remcomparison}]
Write
\[
K(z)=\frac{\mu(z/|z|)}{|z|^{n+2\sigma}},\qquad
K(z)=K(-z),\qquad
\frac{\lambda}{|z|^{n+2\sigma}}\leq K(z)
\leq\frac{\Lambda}{|z|^{n+2\sigma}}.
\]

\textbf{Step 1. Evaluation at a positive upper contact.}
Suppose $\varphi$ touches $u$ from above at $(x_0,t_0)\in\Omega_T$, with $u(x_0,t_0)>0$. As in Definition~\ref{defn:vis}, choose $r_0,\tau_0>0$ such that $\overline{B_{2r_0}(x_0)}\subset\Omega$, $t_0-\tau_0>-T$, 
\[
\varphi\geq u\quad\text{on }B_{r_0}(x_0)\times(t_0-\tau_0,t_0],
\]
and
\[
\varphi(x_0,t_0)=u(x_0,t_0),\qquad
|\delta\varphi(x_0,z,t_0)|\leq C|z|^2
\quad\text{for }|z|<r_0.
\]
For $0<r<r_0$, let $u_{\varphi,r}$ be the patch in Definition~\ref{defn:vis}, using the cylinder $B_r(x_0)\times(t_0-\tau_0,t_0+\tau_0)$. In particular,
\[
u_{\varphi,r}(x,t_0)=
\begin{cases}
\varphi(x,t_0),&x\in B_r(x_0),\\
u(x,t_0),&x\notin B_r(x_0).
\end{cases}
\]
The parabolic subsolution inequality gives
\[
\begin{aligned}
u(x_0,t_0)^{-\gamma}\varphi_t(x_0,t_0)
&\leq Lu_{\varphi,r}(x_0,t_0)\\
&=\int_{|z|<r}\delta\varphi(x_0,z,t_0)K(z)\,dz
+\int_{|z|\geq r}\delta u(x_0,z,t_0)K(z)\,dz.
\end{aligned}
\]
The elliptic assumption $Lu\leq0$ is tested from below, so it does not imply $Lu_{\varphi,r}(x_0,t_0)\leq0$ for this upper patch. We first show that the integral defining $Lu(x_0,t_0)$ converges absolutely.

Since $\varphi\geq u$ near the contact point,
\[
\delta u(x_0,z,t_0)\leq\delta\varphi(x_0,z,t_0)
\leq C|z|^2\qquad (|z|<r_0).
\]
Moreover, \eqref{eq:tail-class} implies
\[
\begin{aligned}
&\int_{|z|\geq r_0}|\delta u(x_0,z,t_0)|K(z)\,dz\leq C_{r_0,x_0}\left(
u(x_0,t_0)+\int_{\R^n}\frac{|u(y,t_0)|}{1+|y|^{n+2\sigma}}\,dy
\right)<\infty.
\end{aligned}
\]
Consequently,
\[
\int_{\R^n}\bigl(\delta u(x_0,z,t_0)\bigr)^+K(z)\,dz<\infty.
\]
For the negative part, the preceding lower bound on $Lu_{\varphi,r}$ yields
\[
\begin{aligned}
&\int_{|z|\geq r}\bigl(\delta u(x_0,z,t_0)\bigr)^-K(z)\,dz\\
&=\int_{|z|\geq r}\bigl(\delta u(x_0,z,t_0)\bigr)^+K(z)\,dz+\int_{|z|<r}\delta\varphi(x_0,z,t_0)K(z)\,dz
-Lu_{\varphi,r}(x_0,t_0)\\
&\leq\int_{\R^n}\bigl(\delta u(x_0,z,t_0)\bigr)^+K(z)\,dz
+Cr^{2-2\sigma}-u(x_0,t_0)^{-\gamma}\varphi_t(x_0,t_0).
\end{aligned}
\]
The right-hand side is bounded independently of $0<r<r_0$. Monotone convergence as $r\downarrow0$ therefore gives integrability of the negative part. Thus $Lu(x_0,t_0)$ is well-defined, and
\[
\begin{aligned}
0&\leq Lu_{\varphi,r}(x_0,t_0)-Lu(x_0,t_0)\\
&=\int_{|z|<r}
\bigl[\delta\varphi(x_0,z,t_0)-\delta u(x_0,z,t_0)\bigr]K(z)\,dz\\
&\rightarrow 0 \quad\mbox{as }r\downarrow0.
\end{aligned}
\]
Passing to the limit in the parabolic inequality gives
\begin{equation}\label{eq:upper-evaluation}
\begin{gathered}
\int_{\R^n}|\delta u(x_0,z,t_0)|K(z)\,dz<\infty,\qquad
\varphi_t(x_0,t_0)\leq u(x_0,t_0)^\gamma Lu(x_0,t_0).
\end{gathered}
\end{equation}

It remains to prove $Lu(x_0,t_0)\leq0$ for this classically defined integral. Set $f=u(\cdot,t_0)$, and choose
\[
\rho\in C_c^\infty(B_1),\qquad \rho\geq0,\qquad
\rho(z)=\rho(-z),\qquad \int_{\R^n}\rho=1.
\]
For $0<\ell<r_0$, write
\[
\rho_\ell(z)=\ell^{-n}\rho(z/\ell),\qquad
f_\ell=f*\rho_\ell.
\]
We first justify
\[
Lf_\ell(x_0)\leq0.
\]
To apply the addition and stability results for bounded functions without assuming that $f$ is globally bounded, set
\[
f_N=\min\{f,N\},\qquad N>\max_{\overline\Omega}f.
\]
Then $f_N=f$ in $\overline\Omega$ and $f_N\leq f$ in $\R^n$. If $\psi$ touches $f_N$ from below at $x\in\Omega$, it also touches $f$ from below there, and the corresponding spatial patches satisfy
\[
(f_N)_\psi\leq f_\psi,\qquad
(f_N)_\psi(x)=f_\psi(x),\qquad
L(f_N)_\psi(x)\leq Lf_\psi(x)\leq0.
\]
Hence $Lf_N\leq0$ in $\Omega$ in the viscosity sense. By translation invariance,
\[
L\bigl(f_N(\cdot-y)\bigr)\leq0
\quad\text{in }B_{r_0}(x_0),\qquad |y|<\ell.
\]
Theorem~5.9 of \cite{luis}, applied to the linear operator $L$ and the singleton class $\{L\}$, implies that finite nonnegative linear combinations preserve this inequality. In particular,
\[
\begin{gathered}
a_j\geq0,\quad \sum_{j=1}^J a_j=1,\quad |y_j|<\ell
\quad \Longrightarrow\quad
L\left(\sum_{j=1}^J a_j f_N(\cdot-y_j)\right)\leq0
\quad\text{in }B_{r_0}(x_0).
\end{gathered}
\]
Approximate $f_N*\rho_\ell$ by these finite averages, locally uniformly in $B_{r_0}(x_0)$ and in the weighted $L^1$ norm. Such an approximation follows from local uniform continuity of $f_N$ in $\Omega$ and continuity of translations in weighted $L^1$. After passing to a subsequence, the convergence also holds almost everywhere in $\R^n$. All these averages are bounded between $0$ and $N$, so Lemma~4.5 of \cite{luis} gives
\[
L(f_N*\rho_\ell)\leq0\quad\text{in }B_{r_0}(x_0).
\]
Since $f_N*\rho_\ell$ is smooth, this inequality holds classically. For fixed $\ell$, we have
\[
f_N*\rho_\ell=f_\ell\quad\text{in }B_{r_0}(x_0),
\]
and
\[
\begin{aligned}
&\int_{\R^n}\frac{|(f_N*\rho_\ell)(y)-f_\ell(y)|}{1+|y|^{n+2\sigma}}\,dy\leq C_\ell\int_{\R^n}\frac{|f_N(y)-f(y)|}{1+|y|^{n+2\sigma}}\,dy
\longrightarrow0\qquad \mbox{as }N\to\infty.
\end{aligned}
\]
Since the two functions agree near $x_0$, the difference of their nonlocal integrals has no singularity there. Therefore,
\[
\begin{aligned}
&\bigl|L(f_N*\rho_\ell)(x_0)-Lf_\ell(x_0)\bigr|\leq C_{r_0,x_0}\int_{\R^n}
\frac{|(f_N*\rho_\ell)(y)-f_\ell(y)|}{1+|y|^{n+2\sigma}}\,dy
\longrightarrow0 \quad \mbox{as }N\to\infty,
\end{aligned}
\]
which proves $Lf_\ell(x_0)\leq0$.

We now let $\ell\downarrow0$. By symmetry and Fubini's theorem,
\[
L\rho_\ell(z)=L\rho_\ell(-z),\qquad
\int_{\R^n}L\rho_\ell(z)\,dz=0,
\]
and hence
\[
\begin{aligned}
Lf_\ell(x_0)
&=\int_{\R^n}f(x_0-z)L\rho_\ell(z)\,dz\\
&=\int_{\R^n}
\left(\frac{f(x_0+z)+f(x_0-z)}2-f(x_0)\right)L\rho_\ell(z)\,dz\\
&=\int_{\R^n}\delta f(x_0,z)L\rho_\ell(z)\,dz.
\end{aligned}
\]
Taylor's formula for $\rho$ and the kernel bound outside its support give
\[
|L\rho(\xi)|\leq C(1+|\xi|)^{-n-2\sigma}.
\]
By homogeneity,
\[
\begin{aligned}
|L\rho_\ell(z)|
&=\ell^{-n-2\sigma}|L\rho(z/\ell)|\leq\frac{C}{(\ell+|z|)^{n+2\sigma}}
\leq\frac{C}{|z|^{n+2\sigma}},\qquad z\ne0.
\end{aligned}
\]
For $|z|>\ell$, the support condition on $\rho_\ell$ gives
\[
L\rho_\ell(z)=\int_{B_\ell}\rho_\ell(y)K(z-y)\,dy
\longrightarrow K(z)
\]
at every Lebesgue point $z\ne0$ of $K$, and hence almost everywhere. Finally,
\[
|\delta f(x_0,z)L\rho_\ell(z)|
\leq\frac{C}{\lambda}|\delta f(x_0,z)|K(z)\in L^1(\R^n)
\]
by \eqref{eq:upper-evaluation}. Dominated convergence yields
\[
Lu(x_0,t_0)
=\int_{\R^n}\delta f(x_0,z)K(z)\,dz
=\lim_{\ell\downarrow0}Lf_\ell(x_0)\leq0.
\]

\textbf{Step 2. Strict perturbation and comparison.}
For $\eta>0$, set
\[
w(x,t)=u(x,t)-\eta(t+T+1).
\]
The function $w$ is continuous on $\overline\Omega\times[-T,0]$ and satisfies \eqref{eq:tail-class}. The initial and exterior inequalities give
\[
\begin{split}
w-v&\leq-\eta\quad\text{on }
(\overline\Omega\times\{-T\})\cup(\partial\Omega\times[-T,0]),\\
w(\cdot,t)-v(\cdot,t)&\leq-\eta
\quad\text{a.e. in }\Omega^c,\qquad t\in[-T,0].
\end{split}
\]
Suppose, to the contrary, that
\[
m:=\max_{\overline\Omega\times[-T,0]}(w-v)>0.
\]
Apply the doubling construction in Section~\ref{appa1} to $w,v$, with $c\equiv1$. Denote the upper and lower tests by $\varphi_w$ and $\varphi_v$, and their patches by $W_\rho$ and $V_\rho$. For all sufficiently small $\epsilon,\tau$, the contacts lie in $\Omega\times(-T,0]$ and
\[
\begin{gathered}
p:=w(x_\epsilon,t_\tau)>q:=v(y_\epsilon,s_\tau)\geq0,\\
\partial_t\varphi_w(x_\epsilon,t_\tau)
=\partial_s\varphi_v(y_\epsilon,s_\tau)
=\frac{t_\tau-s_\tau}{\tau}.
\end{gathered}
\]
The function
\[
\psi(x,t)=\varphi_w(x,t)+\eta(t+T+1)
\]
touches $u$ from above at $(x_\epsilon,t_\tau)$, where
\[
u(x_\epsilon,t_\tau)=p+\eta(t_\tau+T+1)>p>0.
\]
Step~1 therefore gives a well-defined number
\[
\ell:=Lu(x_\epsilon,t_\tau)=Lw(x_\epsilon,t_\tau)\leq0,
\]
and
\[
\frac{t_\tau-s_\tau}{\tau}+\eta
\leq\bigl[p+\eta(t_\tau+T+1)\bigr]^\gamma\ell
\leq p^\gamma\ell.
\]
Here $Lw=Lu$ because the perturbation is spatially constant. At the lower contact of $v$, the supersolution inequality gives
\[
\frac{t_\tau-s_\tau}{\tau}\geq q^\gamma LV_\rho(y_\epsilon,s_\tau).
\]
Also, since $W_\rho$ is an upper patch of $w$,
\[
\begin{aligned}
&LW_\rho(x_\epsilon,t_\tau)-\ell=\int_{|z|<\rho}
\bigl[\delta\varphi_w(x_\epsilon,z,t_\tau)
-\delta w(x_\epsilon,z,t_\tau)\bigr]K(z)\,dz
\geq0.
\end{aligned}
\]
Combining these inequalities, we obtain
\begin{align*}
\eta
&\leq p^\gamma\ell-q^\gamma LV_\rho(y_\epsilon,s_\tau)\\
&=q^\gamma\bigl[\ell-LV_\rho(y_\epsilon,s_\tau)\bigr]
+(p^\gamma-q^\gamma)\ell\\
&\leq q^\gamma\bigl[\ell-LV_\rho(y_\epsilon,s_\tau)\bigr]\\
&\leq q^\gamma\bigl[LW_\rho(x_\epsilon,t_\tau)-LV_\rho(y_\epsilon,s_\tau)\bigr],
\end{align*}
where we used $p>q\geq0$, $\ell\leq0$ and $LW_\rho\geq\ell$.

Since $q$ is nonnegative and uniformly bounded, letting $\rho\downarrow0$ and then $(\epsilon,\tau)\to(0,0)$ along the chosen subsequence, and using \eqref{eq:jump-comparison}, we obtain $\eta\leq0$, a contradiction.

Thus $w\leq v$ on $\overline\Omega\times[-T,0]$, and
\[
u(x,t)-v(x,t)\leq\eta(t+T+1)\leq\eta(T+1).
\]
Letting $\eta\downarrow0$ proves the remark.
\end{proof}

\subsection{Comparison for the weighted extremal equations}
\begin{proof}[Proof of Proposition~\ref{com for weight}]
Subtracting $F(t+T)$ from both functions reduces the proof to the case $F=0$. We first treat $M^-$. Set $v^\eta=v+\eta(t+T+1)$. It satisfies
\[
\partial_tv^\eta-a(x)M^-v^\eta\geq\eta.
\]
Suppose $u-v^\eta$ has a positive maximum and apply the doubling construction. The change of variables $y=a(x)^{1/(2\sigma)}z$ gives
\[
a(x)Lf(x)=\int_{\R^n}\delta f\bigl(x,a(x)^{1/(2\sigma)}z\bigr)
\frac{\mu(z/|z|)}{|z|^{n+2\sigma}}\,dz.
\]
This uses the homogeneity of the stable kernel; the angular density is unchanged by a positive scalar dilation. Consequently $a(x)M^-f$ is the infimum of these integrals.

Apply the viscosity inequalities with the patches of radius $a(x_\epsilon)^{1/(2\sigma)}\rho$ and $a(y_\epsilon)^{1/(2\sigma)}\rho$. Equality of the time derivatives implies
\[
\eta\leq\inf_\mu A_\mu-\inf_\mu B_\mu
\leq\sup_\mu(A_\mu-B_\mu),
\]
where $A_\mu,B_\mu$ are the two patched integrals. The locally Lipschitz function $c=a^{1/(2\sigma)}$ satisfies the hypotheses of \eqref{eq:jump-comparison}. This estimate contradicts the preceding positive lower bound. Thus $u\leq v^\eta$, and then $u\leq v$.

For $M^+$, apply the result just proved to $-v,-u$, using $M^-(-f)=-M^+f$. All continuity, integrability, initial, and exterior assumptions are preserved. This proves the proposition for both extremal operators.
\end{proof}

\bigskip
\bigskip

\noindent M. Huo

\noindent Department of Mathematics, The Hong Kong University of Science and Technology\\
Clear Water Bay, Kowloon, Hong Kong\\[1mm]
Email:  \textsf{mhuoaa@connect.ust.hk}

\bigskip

\noindent T. Jin

\noindent Department of Mathematics, The Hong Kong University of Science and Technology\\
Clear Water Bay, Kowloon, Hong Kong\\
Email: \textsf{tianlingjin@ust.hk}

\bigskip

\noindent J. Xiong

\noindent School of Mathematical Sciences, Laboratory of Mathematics and Complex Systems, MOE\\ Beijing Normal University, 
Beijing 100875, China\\
Email: \textsf{jx@bnu.edu.cn}

\medskip

\end{document}